\documentclass[a4paper,english,11pt]{amsart}

\usepackage[utf8]{inputenc}
\usepackage[T1]{fontenc}
\usepackage{babel}
\usepackage[babel]{csquotes}

\usepackage{amsmath}
\usepackage{amsfonts}
\usepackage{amsthm}
\usepackage{bbm}
\usepackage{amssymb}
\usepackage[top=3.4cm, bottom=3.6cm, left=3.6 cm, right=3.6 cm]{geometry}
\usepackage{mathrsfs}
\usepackage[pdftex, pdfstartview=FitW,colorlinks=true,linkcolor=blue,%
urlcolor=blue,citecolor=blue]{hyperref}
\usepackage{stmaryrd}
\usepackage{cleveref}
\usepackage{mathtools}
\usepackage{soul}
\usepackage{comment}

\usepackage{url}
\usepackage{xcolor}
\usepackage{enumitem}

\newtheorem{th.}{Theorem}[section]
\newtheorem{thm}[th.]{Theorem}

\newtheorem{lemma}[th.]{Lemma} 
 
\newtheorem{proposition}[th.]{Proposition}
\newtheorem{corollary}[th.]{Corollary}

\newtheorem{thmx}{Theorem}

\newtheorem{thmxp}{Theorem}

\theoremstyle{definition}
\newtheorem{definition}[th.]{Definition}

\numberwithin{equation}{section}

\newcommand{\N}{\mathbb{N}}
\newcommand{\Z}{\mathbb{Z}}

\newcommand{\R}{\mathbb{R}}
\newcommand{\G}{\mathcal{G}}
\newcommand{\eps}{\varepsilon}
\newcommand{\kg}{\mathfrak{g}}

\newcommand {\cB} {{\mathcal B}}

\newcommand {\cD} {{\mathcal D}}

\newcommand {\cF} {{\mathcal F}}

\newcommand {\cP} {{\mathcal P}}

\newcommand {\cR} {{\mathcal R}}
\newcommand {\cS} {{\mathcal S}}

\newcommand {\sC} {{\mathscr C}}

\newcommand {\sR} {{\mathscr R}}

\newcommand {\sS} {\mathscr{S}}
\newcommand {\sT} {\mathscr{T}}

\DeclareMathOperator{\supp}{supp}

\DeclareMathOperator{\SL}{SL}

\DeclareMathOperator{\Lie}{Lie}

\DeclareMathOperator{\dist}{dist}

\DeclareMathOperator{\Span}{span}

\newcommand{\1}{{\mathbbm{1}}}
\DeclareMathOperator{\Ad}{Ad}

\DeclareMathOperator{\Aff}{Aff}

\DeclareMathOperator{\Lip}{Lip}
\DeclareMathOperator{\Aut}{Aut}
\DeclareMathOperator{\inj}{inj}

\DeclareMathOperator{\Id}{Id}

\newcommand{\dd}{\,\mathrm{d}}

\newcommand{\E}{\mathbb{E}}
\DeclareMathOperator{\Gr}{Gr}
\DeclareMathOperator{\dang}{d_\measuredangle}
\DeclareMathOperator{\Leb}{Leb}

\renewcommand{\setminus}{\smallsetminus}
\renewcommand{\subset}{\subseteq}

\newcommand{\set}[1]{\{\, #1  \,\}}     
\newcommand{\setbig}[1]{\bigl\{\, #1  \,\bigr\}}     
\newcommand{\setBig}[1]{\Bigl\{\, #1  \,\Bigr\}}

\newcommand{\Lyap}{\ell}

\newcommand {\ttr} {\mathtt{r}}
\newcommand {\ttb} {\mathtt{b}}
\newcommand {\ttm} {\mathtt{m}}

\newcommand{\norm}[1]{\lVert#1\rVert}   

\newcommand{\normBig}[1]{\Bigl\lVert#1\Bigr\rVert}
\newcommand{\abs}[1]{\lvert#1\rvert}    
\newcommand{\abse}[1]{\left\lvert#1\right\rvert}

\title{Khintchine dichotomy for self-similar measures}

\author{Timoth\'ee B\'enard }
\address{CNRS – LAGA, Universit\'e Sorbonne Paris Nord, 99 avenue J.-B.
Cl\'ement, 93430 Villetaneuse}
\email{benard@math.univ-paris13.fr}

\author{Weikun He}
\address{State Key Laboratory of Mathematical Sciences, Academy of Mathematics and System Science, Chinese Academy of Sciences, Beijing 100190, China}
\email{heweikun@amss.ac.cn}
\thanks{W.H. is supported by the National Key R\&D Program of China (No. 2022YFA1007500) and the National Natural Science Foundation of China (No. 12288201).}

\author{Han Zhang}
\address{School of Mathematical Science, Soochow University, Suzhou 215006, China
}
\email{hzhang.math@suda.edu.cn}
\thanks{H.Z. is supported by the startup grant of Soochow University.}


\date{}

\begin{document}

\begin{abstract}
We extend Khintchine's theorem to all self-similar probability measures on the real line. When specified to the case of the Hausdorff measure on the middle-thirds Cantor set, the result is already new and provides an answer to an old question of Mahler. The proof consists in showing effective equidistribution in law of expanding upper-triangular random walks on $\SL_{2}(\R)/\SL_{2}(\Z)$, a result of independent interest.
\end{abstract}
\maketitle  

\setcounter{tocdepth}{1}
\tableofcontents


\section{Introduction} 

A Borel probability measure $\sigma$ on the real line $\R$ is called \emph{self-similar} if it satisfies 
\begin{equation} \label{eq-selfsimilar}
\sigma=\sum_{i=1}^{\ttm} \lambda_i \, \phi_{i\star}\sigma
\end{equation}
for some integer $\ttm\geq 1$, some probability vector $(\lambda_1,\cdots,\lambda_{\ttm}) \in \mathbb{R}_{> 0}^{\ttm}$, and some invertible affine maps  $\phi_{1}, \dotsc, \phi_{\ttm} : \R \to \R$ without common fixed point. 
This includes Hausdorff measures on missing digit Cantor sets.
For example, the one on the middle-thirds Cantor set satisfies \eqref{eq-selfsimilar} with $\lambda_1 =  \lambda_2 = 1/2$ and $\phi_{1} : t \mapsto t/3$ and $\phi_{2} : t \mapsto t/3+2/3$.
Another standard definition of self-similar measures requires that all the maps $\phi_{i}$ are contracting.
We do \emph{not} impose such a condition,  see \S\ref{conventions} for further discussion.

It is particularly intriguing to explore the Diophantine properties of points within the support of a self-similar measure. This research topic was proposed by 
Mahler {in} \cite[Section 2]{Mahler84}, asking how well  irrational numbers in the  middle-thirds Cantor set can be approximated by rational numbers. One approach to framing Mahler's question is by investigating whether Khintchine's theorem extends to the middle-thirds Cantor measure (as asked by Kleinbock-Lindenstrauss-Weiss in \cite[Section 10.1]{KLW04}). 
 
Let us recall the classical Khintchine theorem. 
Here and hereafter, $\psi:\mathbb{N}\to \mathbb{R}_{>0}$ is a  function that will be referred to as an \emph{approximation function}. A point $s\in \R$ is called \emph{$\psi$-approximable} if there exist infinitely many $(p,q)\in \mathbb{Z}\times \mathbb{N}$ such that
\begin{equation} \label{khintchine-equation}
|qs-p|<\psi(q).
\end{equation}
Denote by $W(\psi)$ the set of $\psi$-approximable points in $\mathbb{R}$. 
The classical Khintchine theorem for the Lebesgue measure \cite{Khintchine24,Khintchine26} states that given a non-increasing approximation function $\psi$, the set $W(\psi)$  has  null Lebesgue measure if the series $\sum_{q\in \N} \psi(q)$ is convergent, and full Lebesgue measure otherwise.
 
\bigskip
{In this paper,} we extend  Khintchine's theorem to all self-similar probability measures on $\R$.

\begin{thmx}[Khintchine's {theorem} for self-similar measures]  \label{Kintchine-Cantor} 
Let $\sigma$ be a  self-similar  probability measure on $\R$, let $\psi :\mathbb{N}\to \mathbb{R}_{>0}$ be a non-increasing function. Then 
\begin{equation} \label{K-dichotomy}
\sigma (W(\psi))
= \left\{
    \begin{array}{ll}
        0   & \text{if } \sum_{q \in \N} \psi(q)<\infty, \\
         &\\
        1 & \text{if }\sum_{q \in \N} \psi(q)=\infty.
    \end{array}
\right.
\end{equation}
\end{thmx}

In the divergent case, given a $\sigma$-typical $s\in \R$, we also obtain estimates on the number of solutions $(p,q)$ of the inequality \eqref{khintchine-equation} with bounded $q$, see \eqref{eq:primsol} and \eqref{quant-khint-eq}. 


Let us briefly present the state of the art surrounding Khintchine's theorem on fractals.  

For the convergence part, the case $\psi(q)=1/q^{1+\eps}$ was treated by Weiss \cite{Weiss01} for measures  satisfying certain decay conditions, comprising the case of the middle-thirds Cantor measure.
Weiss' result was later generalized to friendly measures on $\mathbb{R}^d$ for arbitrary positive integer $d$ by Kleinbock-Lindenstrauss-Weiss \cite{KLW04}. See also the related work of Pollington-Velani \cite{PV05} on absolutely friendly measures, and that of Das-Fishman-Simmons-Urba\'nski on quasi-decaying measures \cite{DFSU18, DFSU21}. 

For the divergence part, the case  $\psi(q)=\eps/q$ was treated by Einsiedler-Fishman-Shapira  \cite{EFS11} for missing digit  Cantor measures. Simmons-Weiss \cite{SW19} then significantly generalized their result, promoting it  to  arbitrary self-similar measures on $\mathbb{R}^d$ (along with several refinements). 

All the above works focus on specific approximation functions $\psi$.
Under the sole condition that $\psi$ is non-increasing,  Khalil and Luethi \cite{KL23} were able to extend Khintchine's theorem to self-similar measures $\sigma$ on $\R^d$, provided $\sigma$ has \emph{large  dimension} and the underlying IFS $(\phi_{i})_{1\leq i\leq \ttm}$  is contractive, rational, and satisfies the open set condition. In particular, they derived Khintchine's theorem for one-missing digit Cantor sets in base $5$. With a different approach, based on Fourier analysis, Yu \cite{Yu21} also achieved the convergence part of Khintchine's theorem  for general approximation functions, provided $\sigma$ is a measure with sufficiently \emph{fast average Fourier decay}.  The divergence part was very recently settled by Datta-Jana \cite{DJ24} under similar restrictions,  reaching  some cases that are not covered by \cite{KL23} such as a $3$-missing digit Cantor set in base $450$.

All the aforementioned works impose various constraints for the Khintchine dichotomy \eqref{K-dichotomy} to  hold for a fractal measure. Specifically, none of them establishes \eqref{K-dichotomy} in the case of the  middle-thirds Cantor measure advertised by Mahler. \Cref{Kintchine-Cantor} not only addresses this case,  but also significantly extends beyond it. 



\bigskip

\noindent{\emph{Other related research topics}}. As Mahler pointed out in~\cite{Mahler84}, it is also interesting to investigate \emph{intrinsic} Diophantine approximation on a Cantor set.
This means asking how well points on a fractal set can be approximated by rational points sitting \emph{inside} the fractal set itself.
We refer to recent works~\cite{TWW24,CVY24} and references therein for related research in this direction.

 In addition to fractals, Khintchine's theorem has been extensively studied on submanifolds of $\mathbb{R}^d$. 
Major works in this area include \cite{KM98, VV06, Beresnevich12, BY23}.

\bigskip
{\Cref{Kintchine-Cantor} is derived from an effective equidistribution result in homogeneous dynamics which we  now present.} Consider the real algebraic group $G=\SL_{2}(\R)$, a lattice $\Lambda\subseteq G$, and the  quotient space $X= G/\Lambda$, endowed with the standard hyperbolic metric (\S\ref{conventions}) and the Haar probability measure $m_{X}$.   Write $\inj(x)$ the injectivity radius  at $x\in X$ (\S\ref{conventions}). Denote by $B^\infty_{\infty, 1}(X)$  the set of smooth functions on $X$ which are bounded and have bounded  order-$1$ derivatives, write $\cS_{\infty, 1}(\cdot)$  the associated  $C^1$-norm (\S\ref{conventions}). For $t>0$, $s\in\R$, write $a(t), u(s)\in  G$ the elements given by 
\[a(t)=\begin{pmatrix} t^{1/2}&0\\0&t^{-1/2}\end{pmatrix} \qquad u(s)=\begin{pmatrix} 1&s\\0&1\end{pmatrix}.\]

\begin{thmx}[Effective equidistribution of expanding fractals]  \label{Kintchine-dynamics} 
Let $\sigma$ be a self-similar probability measure on $\R$.
There exists a constant $c=c(\Lambda, \sigma)>0$ such that for all $t>1$, $x \in X$,  $f \in B^\infty_{\infty, 1}(X)$, we have
\begin{equation} \label{eq-thB}
\int_{\R} f (a(t)u(s)x) \dd\sigma(s) = \int_{X} f  \dd m_{X}\,+\, O\bigl(\inj(x)^{-1}\cS_{\infty, 1}(f) t^{-c} \bigr)
\end{equation}
where the implicit constant in $O(\cdot)$ only depends on $\Lambda$ and $\sigma$. 
\end{thmx}

\Cref{Kintchine-dynamics} states the exponential equidistribution of the measure $\sigma$ seen on a piece of horocycle based at $x$ and expanded by the action of the geodesic flow. The exponent $c$ in the rate of equidistribution  is uniform in $x$, however, equidistribution may take more time to start when $x$ is high in the cusp. This is reflected by the term $\inj(x)^{-1}$ in the rate. A refinement of \Cref{Kintchine-dynamics} tackling double equidistribution will also be established, see \Cref{eq:all-range-preview}.  

The link between homogeneous dynamics and Diophantine approximation  is known as Dani's correspondence \cite{Dani85}. In  \cite{KM99}, Kleinbock-Margulis explicitely demonstrated how to use dynamics to obtain a new proof of the classical Khintchine theorem for the Lebesgue measure, see also the variant \cite{Sullivan} by Sullivan and the seminal work of Patterson \cite{Patterson76}. This dynamical perspective laid the foundation for many subsequent works generalizing Khintchine's theorem in various aspects, see e.g. \cite{KSY22,CY24,KL23}.
In particular, the implication from \eqref{eq-thB} to the convergent case of \Cref{Kintchine-Cantor} is given in the work of Khalil-Luethi \cite[Theorem 9.1]{KL23}. 
Under the extra assumptions that $\sigma$ arises from a contractive IFS satisfying the open set condition, they also show that \eqref{eq-thB} is sufficient to establish the divergent case of \Cref{Kintchine-Cantor}, see \cite[Theorem 12.1]{KL23}. Their proof relies on a subtle inverse Borel-Cantelli Lemma. 
Here, we adopt an approach that is closer to Schmidt's original proof of the quantitative Khintchine theorem \cite{Schmidt}. Taking advantage of \Cref{Kintchine-dynamics}, this enables us to get rid of extra assumptions and
has the double advantage of being shorter and quantitative, see \Cref{Sec-Khintchine-dich}.

Besides its applications to Diophantine approximation, \Cref{Kintchine-dynamics} is interesting in its own right. It can be seen as a \emph{fractal} and \emph{effective} version of Ratner's equidistribution theorem for unipotent flows on $X$. Recall that Ratner's  theorem states that any unipotent orbit on a finite-volume homogeneous space  equidistributes within the smallest finite-volume homogeneous subspace that contains it. Unfortunately, the proof gives no information on the rate of equidistribution. Over the past years, substantial efforts were made to obtain an effective version of Ratner's theorem, in other terms quantify the equidistribution of large but bounded pieces of unipotent orbits. In the case where the unipotent orbit arises from  the action of a  horospherical subgroup, Kleinbock and Margulis established in \cite{KM96,KM12}  the effective equidistribution of  expanding translates under the corresponding diagonal flow. 
 More recently, significant progress on effective Ratner was made by Einsiedler-Margulis-Venkatesh \cite{EMV09}, Str\"{o}mbergsson \cite{Strombergsson15}, Kim \cite{Kim24}, Lindenstrauss-Mohammadi \cite{LM23}, Lindenstrauss-Mohammadi-Wang \cite{LMW-EffEq}, Yang \cite{Yang-EffEq} and  Lindenstrauss-Mohammadi-Wang-Yang \cite{LMWY25}. We note that these works focus on the expanding translates of the \emph{Haar measure} on a piece of  unipotent orbit. 
 In \cite{KL23}, Khalil and Luethi  obtain the first effective equidistribution of expanding \emph{fractal measures} on a unipotent orbit in $\SL_{d+1}(\R)/\SL_{d+1}(\Z)$. They argue under the assumption that the underlying IFS is contractive, rational, satisfies the open set condition, and the measure $\sigma$ is thick enough. They also require that the starting point $x$ belongs to a specific countable set related to the IFS.  In Datta-Jana \cite{DJ24}, effective equidistribution for expanding measures are also obtained in $\SL_{2}(\R)/\SL_{2}(\Z)$ assuming sufficiently fast average Fourier decay and restrictions on the starting point $x$.  
  \Cref{Kintchine-dynamics}  generalizes  Khalil-Luethi's and Datta-Jana's equidistribution results in $\SL_2(\R)/\SL_2(\Z)$ in so far as it allows for an arbitrary lattice $\Lambda$,    any starting point $x$, and most importantly any self-similar measure $\sigma$. The dependence of our error term on the starting point  is also more precise.

\medskip
\noindent\emph{Remark}. The weak-$*$ convergence $\lim_{t \to +\infty}a(t)u(s)x \dd\sigma(s) = m_{X}$ resulting from \Cref{Kintchine-dynamics} is also new. Convergence without rate is also addressed in the independent concurrent work of Khalil-Luethi-Weiss~\cite{KLW25} for rational carpet IFS's in all dimensions. Note however that effectivity, and more precisely a polynomial convergence rate as in \eqref{eq-thB}, is crucial to derive the Khintchine dichotomy \eqref{K-dichotomy}  through Dani's correspondence.

\bigskip

\bigskip

We prove \Cref{Kintchine-dynamics} from the point of view of random walks.
The connection between the asymptotic behaviour of an expanding fractal and that of a random walk is rooted in the work of  Simmons-Weiss~\cite{SW19} and  further exploited in \cite{PS20,PSS23,KL23,DGW24,AG24}.
In our paper, this connection takes the form of \Cref{lm:cocycle}.

We establish the following effective equidistribution in law for random walks driven by  expanding  upper triangular matrices on $X $. In the statement below,  $\R^2$ is endowed with the usual Euclidean structure and we write $e_{1}:=(1,0)\in \R^2$.

\begin{thmx}[Effective equidistribution for random walks]
\label{mu^n-equidistribution}
Let $\mu$ be a finitely supported probability measure on the group
\[
\set{a(t)u(s) : t > 0,\, s \in \R} \subset G. 
\]
Assume that the support of $\mu$ is not simultaneously diagonalizable, and $\mu$ satisfies $\int_{G}\log \|ge_{1}\| \dd\mu(g)>0$.
Then there exists a constant $c=c(\Lambda, \mu) > 0$ such that for all $x \in X$, $n\geq 1$ and $f \in B^\infty_{\infty, 1}(X)$,  we have
\[\mu^{*n}*\delta_{x}(f)  =   m_{X}(f)  + O\bigl(\inj(x)^{-1}  \cS_{\infty, 1}(f) e^{-c n}\bigr)\]
where the implicit constant in $O(\cdot)$ only depends on $\Lambda$ and $\mu$. 
\end{thmx}


The proof of \Cref{mu^n-equidistribution} is inspired by \cite{BH24}, where the  first-named and second-named authors establish effective equidistribution for random walks on $X$ which are driven by a \emph{Zariski-dense} probability measure on $G$. In our context however, the acting group is \emph{solvable}. The proof consists of three phases. Each step concerns the dimension of the distribution of the random walk at some scale.

First, we show that the random walks gains some initial positive dimension: there exist constants $\kappa>0$, $A>0$ determined by $\Lambda, \mu$ such that for every small $\rho>0$, $x,y\in X$, every $n\geq  |\log \rho | +A |\log \inj(x)|$, 
\[{\mu^{*n}}*\delta_{x}(B_{\rho}y)\leq \rho^\kappa\]
 where the notation $B_{\rho}y$ refers  to the open ball of radius $\rho$ centered at  $y$ in $X$. 

Second, we bootstrap the value of the exponent $\kappa$ arbitrarily close to $3$, say up to $3-\eps$ provided $\rho\leq \rho_{0}(\eps,\Lambda, \mu)$ and $n \geq C_{0}(\eps, \Lambda, \mu) |\log \rho |+A |\log \inj(x)|$.
The method is based on the multislicing argument from \cite{BH24}, which, in turn, relies on discretized projection theorems à la Bourgain.
The idea of iterating a discretized projection theorem in order to bootstrap (rough) dimension dates back to the work of Bourgain-Furman-Lindenstrauss-Mozes~\cite{BFLM} and played an important role in the most recent advances on effectivizing Ratner's theorem mentioned above.
In a very different context, it also played crucial role in recent developments in projection theory (e.g. Orponen-Shmerkin-Wang~\cite{OSW}).
The way we implement this iteration is different from these works and originates from \cite{BH24}.

Finally, once the dimension is close to be full, we conclude using the spectral gap of the convolution operator $f \mapsto \mu*f$ acting on  $L^2(X)$.

\Cref{Kintchine-dynamics}  follows from \Cref{mu^n-equidistribution}, using \Cref{lm:cocycle} and a probabilistic argument. The convergence part of \Cref{Kintchine-Cantor} is then a direct consequence of \Cref{Kintchine-dynamics}, case $\Lambda=\SL_{2}(\Z)$, and \cite[Theorem 9.1]{KL23}. The divergence part is obtained from a refinement of \Cref{Kintchine-dynamics} about double equidistribution,   inspired by \cite{KSW17}, and builds upon  Schmidt's original proof of the quantitative classical Khintchine theorem \cite{Schmidt}.
\bigskip

\noindent{\bf Allowing $\lambda$ to have infinite support}.
Our method allows for slightly more general statements,  extending the aforementioned Khintchine dichotomy and  equidistribution results to   measures arising from a randomized IFS with potentially \emph{infinite support}, provided a finite exponential moment. 

We let $\Aff(\R)$ denote the affine group of $\R$. For every $\phi\in \Aff({\R})$, we let $\ttr_{\phi} \in \R^*$, $\ttb_{\phi}\in \R$ denote the unique numbers such that 
\begin{equation} \label{def-rphi-bphi}
\phi(t)=\ttr_{\phi}t+\ttb_{\phi}, \quad \forall t\in \R.
\end{equation}
We say a probability measure $\lambda$ on  $\Aff(\R)$ has a finite exponential moment if there exists $\eps>0$ such that 
\begin{equation} \label{def-expmoment-lambda}
\int_{\Aff(\R)} |\ttr_{\phi}|^\eps+ |\ttr_{\phi}^{-1}|^\eps+ |\ttb_{\phi}|^\eps \dd \lambda(\phi) <\infty.
\end{equation}

\begin{thmxp}
\label{ThA'}
Let $\lambda$ be a probability measure on $\Aff(\R)$ with a finite exponential moment and such that $\supp \lambda$ does not have a global fixed point. Let $\sigma$ be a probability measure on $\R$ satisfying $\lambda*\sigma=\sigma$.  Then $\sigma$ satisfies the Khintchine dichotomy \eqref{K-dichotomy}.
\end{thmxp}

\begin{thmxp}
\label{ThB'}
Under the same assumptions,  $\sigma$ satisfies the effective equidistribution for expanding translates from \Cref{eq-thB}. 
\end{thmxp}

Recall that a probability measure $\mu$ on $G$ has a finite exponential moment if for some $\eps>0$, we have 
\begin{equation} \label{def-expmoment-mu}
\int_{G} \|g\|^\eps \dd \mu(g) <\infty.
\end{equation}

\begin{thmxp}
\label{ThC'}
\Cref{mu^n-equidistribution} is valid when the finite support assumption on $\mu$ is relaxed into a finite exponential moment condition. 
\end{thmxp}

\bigskip
\noindent{\bf Structure of the paper}.
In \Cref{Sec-preliminary}, we fix notations for the rest of the paper, we present moment and non-concentration estimates for self-similar measures, and we recall some recurrence properties of the $\mu$-walk on $X$.
In \Cref{Sec-positive-dimension}, we deduce  positive dimension of ${\mu^{*n}}*\delta_{x}$ at exponentially small scales.
In \Cref{Sec-bootstrap}, we bootstrap the dimension until it reaches a number arbitrarily close to $3 = \dim X$.
In \Cref{Sec-equidistribution}, we deduce the equidistribution statements, namely 
\Cref{ThB'} and \Cref{ThC'}. 
In \Cref{Sec-double-eq}, we upgrade  \Cref{ThB'} into a double equidistribution statement.
In \Cref{Sec-Khintchine-dich}, we prove the Khintchine dichotomy for every probability measure on $\R$ satisfying certain equidistribution properties, yielding in particular  \Cref{ThA'}.

\bigskip
\noindent{\bf Acknowledgements}.
The authors thank Nicolas de Saxc\'e for sharing his insight on random walks and Diophantine approximation, as well as Tushar Das, Shreyasi Datta, Larry Guth, Osama Khalil, Dmitry Kleinbock, Manuel Luethi, David Simmons, Sanju Velani and the anonymous referee for many helpful comments on earlier versions of this paper.
W.H. and H.Z. thank Barak Weiss for enlightening discussions. H.Z. thanks Ronggang Shi for his encouragement.

\section{Preliminaries}  \label{Sec-preliminary}

In this section, we set up notations and collect basic facts that will be useful for the rest of the paper. 

\subsection{Notation and Conventions} \label{conventions}
Throughout this paper, $G = \SL_2(\R)$, $\Lambda\subseteq G$ is a lattice, and $X= G/\Lambda$.

\bigskip
\noindent{\bf Metric.}
We fix a basis $(e_{-}, e_{0}, e_{+})$ of the Lie algebra $\kg=\Lie( G)$  given by
\[
e_{-}= \begin{pmatrix} 0&0\\ 1 &0\end{pmatrix}  \,\, \,\, \,\, e_{0}=\begin{pmatrix} 1 &0 \\  0&-1 \end{pmatrix}\,\, \,\,   \,\, e_{+}=\begin{pmatrix} 0&1 \\ 0 &0\end{pmatrix}
\]
We assume throughout that $ G$ is endowed with the unique right-invariant Riemannian metric for which $(e_{-}, e_{0}, e_{+})$ is orthonormal. This induces a distance on $ G$ and the quotient $X$ that we denote by $\dist$ in both cases. Given $\rho>0$, we write $B_{\rho}$ to denote  the open ball of radius $\rho > 0$ centered at the neutral element $\Id$ in $G$.
Then the open ball of  radius $\rho$ centered {at} a point $x \in X$ coincides with $B_\rho x$.

The \emph{injectivity radius} of $X$ at a point $x$ is  
\[\inj(x)=\sup\{\, \rho>0 \,:\,\text{the map $B_{\rho}\rightarrow X, g\mapsto gx$ is injective} \,\}. \]

\bigskip
\noindent{\bf  Sobolev norms.}
Set $\Xi_{l}$ the words on the alphabet $\{e_{-},e_{0}, e_{+}\}$ of length at most $l$. Each $D\in \Xi_{l}$ acts as a differential operator on the space of  smooth functions $C^\infty(X)$.  Given $f\in C^\infty(X)$, $k,l\in \N \cup \{\infty\}$, we set
\[\cS_{k,l}(f)=\sum_{D\in \Xi_{l}} \|Df\|_{L^k},\]
where $\|\cdot \|_{L^k}$ refers to the $L^k$-norm for the Haar probability measure on $X$. We let  $B^{\infty}_{k,l}(X)$ denote the space of  smooth functions $f$ on $X$ such that $\cS_{k,l}(f)<\infty$.  


\bigskip
\noindent{\bf Haar measure.}
Let $m_G$ denote the Haar measure on $G$ normalized so that the $G$-invariant Borel measure $m_X$ it induces on $X$ is a probability measure.


\bigskip
\noindent{\bf Driving measures $\lambda$ and $\mu$.}
Let $\Aff(\R)^+$ denote the group of orientation preserving affine transformations of the real line.
Denote by
\[
P = \set{a(t)u(s) : t > 0,\, s \in \R}\subset G
\]
the subgroup of upper triangular matrices with positive diagonal entries.
For every $g \in P$, we let $\ttr_g \in \R_{> 0}$ and $\ttb_g \in \R$ be the unique numbers such that 
\[
g = a(\ttr_g)^{-1} u(\ttb_g) =
\begin{pmatrix} \ttr_g^{-1/2}& \ttr_g^{-1/2}\ttb_g\\0&\ttr_g^{1/2}\end{pmatrix}.
\]
We identify $P$ with $\Aff(\R)^+$ by mapping $g \in P$ with the similarity $s \mapsto \ttr_g s + \ttb_g$.
This is an anti-isomorphism between the two groups.

Fix  a probability measure $\lambda$ on $\Aff(\R)^+$ with support $\supp \lambda$, denote by  $\mu$ the corresponding probability measure on $P$ via the above anti-isomorphism. Throughout this paper, $\lambda$ and $\mu$ determine each other in this way. For  $n \in \N$, we write $\lambda^{*n} = \lambda * \dotsm * \lambda$ to denote the $n$-fold convolution of $\lambda$ with itself, we define $\mu^{*n}$ similarly.

We assume that $\lambda$, and equivalently $\mu$, has a \emph{finite exponential moment} \eqref{def-expmoment-mu}. With our notations, this means there exists $\eps>0$ such that 
\[
\int_{P} |\ttr_{g}|^\eps+ |\ttr_{g}^{-1}|^\eps+ |\ttb_{g}|^\eps \dd \mu(g) <\infty.
\]

We assume that \emph{$\supp \lambda$ does not have a {global} fixed point in $\R$}. This amounts to saying that $\supp \mu$ does not have two common fixed points on the projective line, or alternatively, that the matrices in $\supp \mu$ are not {simultaneously} diagonalizable.

\bigskip
\noindent{\bf Self-similar measure $\sigma$.}
Throughout this paper, we let $\sigma$ denote a probability measure on $\R$ that is $\lambda$-stationary, which means 
\[
\sigma = \int_{\Aff(\R)}  \phi_{\star} \sigma \dd\lambda(\phi).
\]
By a theorem of Bougerol-Picard \cite[Theorem 2.5]{BP92}, the existence of such $\sigma$ is equivalent to the condition:
\begin{equation} \label{eq-contractive-average}
\int_{P} \log \ttr_g \dd \mu(g) <0,
\end{equation}
i.e. the random walk on $\R$ driven by $\lambda$ is \emph{contractive in average}.  Moreover, provided existence, the measure $\sigma$ is uniquely determined by $ \lambda$, see \cite[Corollary 2.7]{BP92}.

\bigskip
\noindent{\bf Lyapunov exponent.} 
Let $\Ad:G\rightarrow \Aut(\kg)$  be the adjoint representation.
We denote by $\Lyap$  the top Lyapunov exponent associated to $\Ad_\star\mu$. 
It is determined only by the diagonal terms and is equal to
\begin{align}\label{def-Lyap}
\Lyap =  - \int_{P} \log \ttr_g \dd \mu(g)>0.
\end{align}


\bigskip
\noindent{\bf Asymptotic notations.} We use the Landau notation $O(\cdot)$ and the Vinogradov symbol $\ll$. Given $a,b>0$, we also write $a\simeq b$ for $a\ll b\ll a$. We also say that a statement involving $a,b$  is valid under the condition $a\lll b$ if it holds provided  $a\leq \eps b$ where $\eps>0$ is a small enough constant. The asymptotic notations $O(\cdot)$, $\ll$, $\simeq$, $\lll$ {implicitly} refer to constants that \emph{are allowed to depend on the lattice $\Lambda$, and the measure $\lambda$} (or equivalently on $\mu$ as one determines the other by our conventions).
The dependence in other parameters will appear in subscript.

\subsection{Regularity of self-similar measures}
We first recall that the measure $\sigma$ has finite moment of positive order and is H\"older-regular. We often refer to the second property as having positive dimension (at all scales).

\begin{lemma}[Moment and H\"older-regularity of $\sigma$]  \label{sigma Holder}
There exists $\gamma>0$ such that 
\begin{equation*}
(i)\,\,\int_{\R} |s|^\gamma\dd\sigma(s)<\infty, \qquad \quad (ii)\,\,\forall r>0, \,\sup_{s\in \R}  \sigma(s+[-r, r]) \ll  r^\gamma.
\end{equation*}
\end{lemma}
\begin{proof}
Item (i) follows from Kloeckner \cite[Theorem 3.1]{Kloeck22}  and item (ii) is a consequence of \cite[Theorem 2.12]{AG20} due to Aoun and Guivarc'h.
If one is only interested in self-similar measures arising from finitely supported contractive IFS's, then (i) is trivial because $\sigma$ has compact support in this case, and a short proof of item (ii) can be found in a work of Feng--Lau~\cite[Proposition 2.2]{FengLau}.
\end{proof}

Given an integer $n \in \N$, denote by $\sigma^{(n)}$ the image measure of $\mu^{*n}$ under the map $g \in P \mapsto \ttb_g \in \R$.
Equivalently, $\sigma^{(n)}={\lambda^{*n}}*\delta_{0}$, where $\delta_0$ denotes the Dirac measure at $0\in \mathbb{R}$. 
We show that the measures $\sigma^{(n)}$ have a uniformly finite positive moment, and uniform positive dimension  above an exponentially small scale.  For this, we first observe that $\sigma^{(n)}$ converges toward $\sigma$ at exponential rate. We denote by $\Lip(\R)$ the space of bounded Lipschitz functions on $\R$ with the norm 
$ \|f\|_{\Lip}=\|f\|_{\infty}+\sup_{s\neq t}\frac{|f(s)-f(t)|}{|s-t|}$.

\begin{lemma} \label{eq-expconv}
There exists $\eps>0$ such that for all $n\geq 0$, all $f \in \Lip(\R)$, we have
\[|\sigma^{(n)}(f)-\sigma(f)| \ll e^{-\eps n} \|f\|_{\Lip}.\]
\end{lemma}

\begin{proof}
We may assume $ \|f\|_{\Lip}=1$. Then we have, 
\begin{align} \label{eq-expconv1}
|\sigma^{(n)}(f)-\sigma(f)| &= |\lambda^{*n}*\delta_{0}(f)-\lambda^{*n}*\sigma(f) \nonumber|\\
& \leq \int_{\Aff(\R)^+\times \R}|f(\phi(0))-f(\phi(s))| \dd(\lambda^{*n}\otimes\sigma)(\phi,s)  \nonumber\\
& \leq \int_{\Aff(\R)^+\times \R} \min(2, \, \ttr_{\phi} |s|) \dd(\lambda^{*n}\otimes\sigma)(\phi,s) ,
\end{align}
where $\ttr_{\phi}>0$ denotes the dilation factor in the affine map $\phi$, see \eqref{def-rphi-bphi}.
Using the principle of large deviations and that $\lambda$ is contracting in average \eqref{eq-contractive-average}, we have for $\eps\lll 1$, 
\begin{align} \label{eq-expconv2}
\lambda^{*n}\set{\phi\,:\,\ttr_{\phi} >e^{-\Lyap n/2}} \ll e^{-\eps n}. 
\end{align}
On the other hand, up to taking smaller $\eps$, \Cref{sigma Holder}(i) guarantees that 
\begin{align} \label{eq-expconv3}
\sigma \set{s\,:\,|s|> e^{\Lyap n/4}} \ll e^{-\eps n}. 
\end{align}
The claim follows from the combination of \eqref{eq-expconv1}, \eqref{eq-expconv2}, \eqref{eq-expconv3}.
\end{proof}

We now deduce our claim on the measures $\sigma^{(n)}$.
\begin{lemma}[Moment and H\"older-regularity of $\sigma^{(n)}$] \label{posdim-sigma}
There exists $\gamma>0$ such that 
\[
(i)\quad \sup_{n\geq 1}\int_{\R} |s|^\gamma \dd\sigma^{(n)}(s)<\infty
\]
and
\[
(ii)\quad \forall n\geq 1,\, \forall r>e^{-n}, \quad \sup_{s\in \R} \sigma^{(n)}(s+[-r, r]) \ll r^\gamma.
\]
\end{lemma}

\begin{proof} Fix $\gamma, \eps\in(0,1)$ as in \Cref{eq-expconv}. 

For (i), we need to check that the map $t\mapsto \sup_{n}\sigma^{(n)}\set{s : |s| \geq t}$ has polynomial decay as $t\to +\infty$.
Given $t > 2$, \Cref{eq-expconv} and \Cref{sigma Holder}(i) imply that for every $n\geq 0$, one has 
\[\sigma^{(n)}\set{s\,:\,|s| \geq t} \leq \sigma \set{ s \,:\,|s| \geq t - 1 } +O(e^{-\eps n}) \ll t^{-\gamma} + e^{-\frac{\eps}{2} n}.\]
Let $R>0$ be a parameter. The above justifies that uniformly in $n$, we have polynomial decays of tail probabilities of $\sigma^{(n)}$ for $t\leq e^{Rn}$. Taking $R \ggg 1$, the exponential moment assumption on $\lambda$ takes care of the case $t> e^{Rn}$, using the observation
\[\sigma^{(n)}\set{s\,:\,|s| \geq t} \leq \lambda^{\otimes n} \Bigl\{ (\phi_{1}, \dots, \phi_{n}) \,:\, n\prod_{k=1}^n \max(1,\ttr_{\phi_k}, |\ttb_{\phi_k}|) \geq t \Bigr\}\]
and the Markov inequality. This justifies (i), with a potentially smaller value of $\gamma$.

Let us check (ii). For $n\geq 0$, $s\in \R$ and $r\geq e^{-\frac{\eps}{2} n}$, \Cref{eq-expconv} guarantees 
\[\sigma^{(n)}([s-r, s+r])\leq \sigma([s-2r, s+2r])+O(e^{-\frac{\eps}{2} n})\ll r^\gamma,\]
whence the claim (with $\frac{\eps}{2}\gamma$ in place of $\gamma$ to treat all scales above $e^{-n}$). 
\end{proof}

Finally, we derive from \Cref{posdim-sigma} that $\sigma^{(n)}$  satisfies a non-concentration estimate with respect to polynomials of degree $2$. 

\begin{lemma}[Regularity of $\sigma^{(n)}$ for quadratic polynomials] \label{computing-non-concentration}
There exists $\gamma>0$ such that for every $n\geq 1$,  $r>e^{-n}$ and  $(a, b, c) \in \R^3$ with $\max(|a|,|b|, |c|)\geq 1$, we have 
\[\sigma^{(n)}\{s: |as^2+bs+c|\leq r\}\ll r^{\gamma}.\]
\end{lemma}

\begin{proof} We may suppose $r\in (0, 1/10)$. 

Assume first $\max(|a|,|b|) <r^{1/8}$. We must have $|c|\geq  1$, so the inequality $ |as^2+bs+c|\leq r$ implies 
$$ |as^2+bs| \geq 1/2$$
and the claim follows by \Cref{posdim-sigma} (i).

Assume now   $\max(|a|,|b|) \geq r^{1/8}$. We first check that the set $E:=\{s \in [-r^{-1/4},r^{-1/4}]\,:\,|as^2+bs+c| \leq r\}$ is included in at most two balls of radius $8 r^{1/8}$. Indeed, 
if $s_{1}, s_{2}\in E$, then $|as_{1}^2+bs_{1}-as_{2}^2 - bs_{2}|\leq 2r$, i.e. 
\[\abs{(s_{1}-s_{2})(b +a(s_{1}+s_{2}))} \leq 2r.\]
Then either $|s_{1}-s_{2}|\leq 2r^{1/2}$ or $|b +a(s_{1}+s_{2})| \leq r^{1/2}$. In the second case, the condition $\max(|a|,|b|)\geq r^{1/8}$ forces $|a|\geq r^{3/8}/4$, then $s_{1}$ belongs to the ball of radius $8 r^{1/8}$ and center $(-a^{-1}b-s_{2})$, hence the claim about $E$. From there, the lemma follows using \Cref{posdim-sigma} (i), (ii).
\end{proof}

\subsection{Recurrence of the random walk} \label{Sec-recurrence}

We recall the following result of non-escape of mass for the $\mu$-walk on $X$.
\begin{proposition}[Effective recurrence on $X$] \label{effective-recurrence} There exist constants $c, c'>0$ depending on $\mu$ only such that for every $x\in X$, $n \in \N$, and $\rho>0$, we have 
\[\mu^{*n}*\delta_{x} \{ \inj < \rho \}\ll (\inj(x)^{-c} e^{- c'n} +1)\rho^c. \]
\end{proposition}

 For  walks on homogeneous spaces, results of this type originate from the work of Eskin-Margulis-Mozes \cite{EMM98} on the quantitative Oppenheim conjecture. They are now understood in the context of semisimple random walks \cite{EM04, BQfv, BS22}, and more generally expanding random walks \cite{PSS23}.
\Cref{effective-recurrence} can be regarded as a consequence of \cite[Proposition 3.3 and Theorem 6.1]{PSS23} combined with some well-known arguments.

In this subsection,  we give a self-contained and more direct proof of \Cref{effective-recurrence}.

\begin{lemma}[Zassenhaus neighborhood]\label{lm:Margulis}
There exists an absolute constant $\eta > 0$ such that for every discrete subgroup $\Lambda' \subset G$, the intersection $B_\eta \cap \Lambda'$ generates a cyclic group.
\end{lemma}
Recall a group is \emph{cyclic} if it is generated by a single element.

\begin{proof} Let $\eta>0$.  Let $g,h\in B_{\eta}$ with $g\neq \Id$. Provided $\eta \lll 1$, we can write $g=\exp(v)$, $h=\exp(w)$ for some $v,w\in B^{\kg}_{2\eta}$ and the Baker-Campbell-Hausdorff formula gives 
\begin{equation*}\label{eq:BCH}
g h g^{-1} h^{-1} = \exp([v,w] +z)
\end{equation*}
where $z\in \kg$  satisfies $\|z\|\ll \eta \|[v,w]\|$. This implies that for $\eta\lll1$, 
\begin{align}
&\dist(g h g^{-1} h^{-1}, \Id) \ll \|v\|\|w\| <  \dist(g, \Id), \label{eq:ghgh}\\
&g h = hg \iff [v,w]=0 \iff w\in \R v. \label{eq:gh=gh}
\end{align}
where the second equivalence in \eqref{eq:gh=gh} is a straightforward computation in $\kg$.

Now let us check that $B_\eta \cap \Lambda'$ generates a cyclic group. Clearly we may assume $B_\eta \cap \Lambda' \neq \{ \Id \}$.
Then by discreteness, we may consider  an element $\gamma = \exp(v) \in B_\eta \cap \Lambda' \setminus \{\Id\}$ minimizing $\dist(\gamma, \Id)$.
By \eqref{eq:ghgh}, for any $h \in B_\eta \cap \Lambda'$, the commutator $\gamma h \gamma^{-1} h^{-1} \in \Lambda'$ is closer to $\Id$ than $\gamma$, hence it must be $\Id$
by  minimality of $\gamma$.
By \eqref{eq:gh=gh}, we infer $h = \exp(t v)$ for some $t \in \R$. If $t\notin \Z$, then $\Lambda'$ contains  an element of the form $\exp(s v)$ where $s\in (0, 1/2]$ which contradicts the minimality of $\gamma$ (say for $\eta\lll1$). 
Therefore $t \in \Z$ and this finishes the proof.
\end{proof}

Using \Cref{lm:Margulis},  we show that for small $c>0$, the function $\inj^{-c} \colon X \to \R_{>0}$ is uniformly contracted under the random walk. Standard  terminology  then qualifies $\inj^{-c}$ as a Margulis function.
 
\begin{lemma}\label{lm:CH}
For  $c \lll 1$, there exist $m \in \N_{> 0}$, $a \in (0,1)$ and $b \in \R_{> 0}$ such that
\[
\forall x \in X,\quad \mu^{*m}*\delta_x (\inj^{-c}) \leq a \inj^{-c}(x) + b.
\]
\end{lemma}

To prepare the proof, we introduce for every parameter $c>0$ the notation
\[
M_c(\mu):=\int_{G} \norm{\Ad(g)}^c \dd \mu(g).
\]
The finite exponential moment assumption on $\mu$ means that $M_c(\mu)$ is finite for $c \lll 1$.

We also observe that for every $g \in G$, the left multiplication on $X$ by $g$ is $\|\Ad(g)\|$-Lipschitz. Using $\|\Ad(g)\|=\|\Ad(g^{-1})\|$,  it follows that: $\forall g\in G, x\in X$,  
\begin{equation}\label{eq:injLip}
  \norm{\Ad(g)}^{-1} \inj(x)\leq \inj(gx) \leq \norm{\Ad(g)} \inj(x).
\end{equation}

\begin{proof}
Let $\eta > 0$ be small enough so that \Cref{lm:Margulis} holds for $B_\eta$ and additionally the logarithm map is well defined and $2$-bi-Lipschitz from $B_\eta$ to a neighborhood of $0$ in $\kg$.
Consider some parameters $c>0$, $m \in \N^*$ and $R > 1$, to be specified later. 

For all $x \in X$ with $\inj(x) \geq R^{-m}\eta/8$, we have by \eqref{eq:injLip} and the submultiplicativity of the norm that
\begin{equation}  \label{contract-b}
\mu^{*m}*\delta_{x}(\inj^{-c}) \leq b := 8^c M_c(\mu)^m R^{mc}\eta^{-c} .
\end{equation}
We will show that for $c\lll1$ and appropriate choices of $m$ and $R$, there is $a \in (0,1)$ such that for all $x \in X$ with $\inj(x) < R^{-m} \eta / 8$, we have
\begin{equation} \label{eq:ainjc}
\mu^{*m}*\delta_{x}(\inj^{-c}) \leq a \inj^{-c}(x).
\end{equation}
Note that \eqref{contract-b} and \eqref{eq:ainjc} together yield the desired contraction property.

We first replace $\inj(x)$ by the norm of a suitable vector in $\kg$. Namely, for every $x = h \Lambda \in X$ with $\inj(x) < \eta/2$, the set
\[
\set{g \in B_\eta \,:\, gx = x} = B_\eta \cap h \Lambda h^{-1}
\]
generates a cyclic group (\Cref{lm:Margulis}).
Let $v_x$ be the logarithm of a generator of this subgroup.
It is uniquely defined up to a minus sign and, using  $\inj(x) < \eta/2$, we have
\[
\frac{1}{4}\norm{v_x} \leq \inj(x) \leq 4\norm{v_x}.
\]

Let $x \in \{ \inj(x)< R^{-m}\eta/8 \}$.
By \eqref{eq:injLip}, we have $\inj(g x) < \eta/8$ whenever 
\[
g \notin E:=\setbig{g\in G \,:\,\norm{\Ad(g)} > R^m}.
\]
We claim that for such $g$,  
\begin{equation} \label{vgx}
v_{gx} = \pm \Ad(g) v_x.
\end{equation}
Indeed, we have
\(
\exp(\Ad(g) v_x) gx = gx
\)
as well as 
\[ 
\dist(\exp(\Ad(g)v_x),\Id) \leq  \norm{\Ad(g) v_x} \leq  \norm{\Ad(g)}\norm{v_x} \leq 8 R^m \inj(x) < \eta.
\]
Hence by the definition of $v_{gx}$, there exists $k \in \Z \setminus \{0\}$ such that
\(
\Ad(g) v_x = k v_{gx}.
\)
Then we also have
\(
\exp(k^{-1}v_x)x = \exp(\Ad(g^{-1}) v_{gx}) x = x,
\)
as well as
\[
\dist(\exp(k^{-1}v_x),\Id) \leq  \norm{k^{-1}v_x} < \eta.
\]
Hence
\(
k^{-1} v_x \in \Z v_x,
\)
then $k \in \{ \pm 1 \}$, yielding \eqref{vgx}.

Recalling the Lyapunov exponent $\Lyap$ from \eqref{def-Lyap}, set
\[
F_{x}:=\setbig{g\in G \,:\,\norm{\Ad(g)v_{x}} <e^{m\Lyap/4}\norm{v_x}}.
\]
Then for every $g \notin E \cup F_x$,
\[
\inj(gx) \geq \frac{\norm{v_{gx}}}{4} = \frac{\norm{ \Ad(g) v_x }}{4} \geq \frac{e^{m\Lyap/4} \norm{v_x}}{4} \geq \frac{e^{m\Lyap/4} \inj(x)}{4^2}.
\]
On the other hand, for $g \in E\cup F_x$, we bound $\inj(gx)$ from below using \eqref{eq:injLip}.
These two lower bounds yield
\begin{equation} \label{eq:fracinjinj}
\frac{\mu^{*m}*\delta_{x}(\inj^{-c})}{\inj^{-c}(x)}
\leq \int_{E\cup F_{x}} \norm{\Ad(g)}^c \dd\mu^{*m}(g) + 4^{2c} e^{-m\Lyap c/4}.
\end{equation}
Using the Cauchy-Schwarz inequality and the submultiplicativity of the norm, we have
\begin{equation} \label{dom-intEF}
\int_{E\cup F_{x}} \norm{\Ad(g)}^{c} \dd\mu^{*m}(g)\leq \mu^{*m}(E\cup F_{x})^{1/2} M_{2c}(\mu)^{m/2},
\end{equation}

We claim that for some $\alpha=\alpha(\mu)>0$, and up to taking parameters $m,R\ggg 1$, we have 
\begin{equation} \label{EF}
\mu^{*m}(E\cup F_{x}) \leq e^{-m\alpha}.
\end{equation}
Note that together with \eqref{eq:fracinjinj} and \eqref{dom-intEF}, this yields the  inequality \eqref{eq:ainjc} with constant 
$a := e^{-m \alpha/2} M_{2c}(\mu)^{m / 2} + 4^{2c} e^{-m\Lyap c/4}$. As desired, we have $a \in (0,1)$ provided $c\lll 1$ and $m\ggg_{c} 1$. 

It remains to show \eqref{EF}.
First,  the Markov inequality yields for all $\eps>0$,
\[\mu^{*m}(E) \leq \left(M_{\eps}(\mu)R^{-\eps}\right)^m\]
whence the claim on $\mu^{*m}(E)$ by choosing  $0<\eps\lll 1$ and  $R\ggg_{\eps} 1$.
We now bound $\mu^{*m}(F_{x})$. Recall the basis $(e_+, e_0, e_-)$ of $\kg$ from \S\ref{conventions}.
Let $e_+^* \colon \kg \to \R$ be the corresponding linear form in the dual basis.
For  $g = a(\ttr^{-1}_g) u(\ttb_g) \in F_x$, we have
\[
\ttr_g^{-1} \abse{e_+^*\bigl( \Ad(u(\ttb_g)) v_x\bigr)} = \abse{e_+^*\bigl( \Ad(g) v_x\bigr)} < e^{m\Lyap/4} \norm{v_x}.
\]
Hence either $\ttr_g > e^{-m\Lyap/2}$ or $\abse{e_+^*\bigl( \Ad(u(\ttb_g)) v_x\bigr)} < e^{-m\Lyap/4} \norm{v_x}$.
By the large deviation principle for $\log \ttr_g$, there is some $\alpha = \alpha(\mu) > 0$ such that
\[
\mu^{*m} \set{g \in G \,:\, \ttr_g > e^{-m\Lyap/2} } \ll e^{- m \alpha}.
\]
It remains to bound
\begin{align*}
    \mu^{*m} \set{g \in G \,:\, \abse{e_+^*\bigl( \Ad(u(\ttb_g)) v_x\bigr)} < e^{-m\Lyap/4} \norm{v_x} }.
\end{align*}
Write $w=v_x/\norm{v_x}= t_{-} e_-+ t_{0}e_0+ t_{+}e_+$ where $t_{-},t_{0},t_{+}\in \R$. Note that for every $s\in \R$, we have 
$$e_+^*(\Ad(u(s)) w) = -t_{-} s^2  - 2t_{0}s+ t_{+},$$
and the variable  $(\ttb_g)_{g\sim \mu^{*m}}$ has law $\sigma^{(m)}$.  Invoking \Cref{computing-non-concentration}, we deduce
\begin{align*}
    \mu^{*m} \set{g \in G \,:\, \abse{e_+^*\bigl( \Ad(u(\ttb_g)) v_x\bigr)} < e^{-m\Lyap/4} \norm{v_x} } \ll e^{- m \alpha}
\end{align*}
up to taking smaller $\alpha=\alpha(\mu)$. This finishes the proof of \eqref{EF}, and of the lemma. 
\end{proof}

 Effective recurrence  now follows from \Cref{lm:CH} and the Markov inequality.

\begin{proof}[Proof of \Cref{effective-recurrence}] Fix parameters $(c, m,a,b)$ as in \Cref{lm:CH} and such that $M_{c}(\mu)<\infty$. Set $b':=b/(1-a)$. 
By iterating the inequality of \Cref{lm:CH}, we obtain  for all $q \in \N$, $x\in X$, 
\[
 \mu^{*q m}*\delta_x (\inj^{-c}) \leq a^{q} \inj^{-c}(x) + b'.
\]
It follows from the Markov inequality that for all $\rho>0$, 
\begin{align} \label{eq-contr-q}
 \mu^{* q m} * \delta_x \{\inj < \rho\} \leq \bigl( a^q \inj(x)^{-c} + b'\bigr)\rho^c .
\end{align}

Now, given $n \in \N$, write $n = q m + k$ with $q \in \N$ and $0 \leq k < m$.
It follows from \eqref{eq-contr-q} and \eqref{eq:injLip} that for all  $x \in X$, $\rho > 0$, 
 \begin{align*}
\mu^{*n} * \delta_x \{ \inj < \rho \} & = \int_G \mu^{* q m} * \delta_{gx} \{ \inj < \rho \} \dd \mu^{* k}(g) \\
& \leq \left(a^q \int_G \inj(gx)^{-c} \dd \mu^{*k}(g) + b'\right) \rho^c \\
& \leq \left(a^q M_c(\mu)^m \inj(x)^{-c} + b'\right) \rho^c.
\end{align*}
This finishes the proof of  effective recurrence.
\end{proof}

\section{Positive dimension} \label{Sec-positive-dimension}

{We show that the $n$-step distribution of the $\mu$-walk starting from a point $x$ acquires positive dimension at an exponential rate, tempered by the possibility that $x$ may be high in the cusp.}

\begin{proposition}[Positive dimension] \label{positive-dimension}
There exists $A, \kappa>0$ such that for every $x\in X$, $\rho>0$,  $n\geq  |\log \rho|+ A |\log \inj(x)|$, we have 
\begin{equation}\label{eq-posdim-c}
 \forall y\in X,\quad \mu^{*n}*\delta_{x}(B_{\rho}y)  \ll \rho^\kappa.
 \end{equation}

\end{proposition}

\begin{proof}
Let $\kappa>0$ be a parameter to specify below. Let   $\rho\in (0, 1/10)$, $n\geq |\log \rho|$,  $x,y\in X$,  and assume 
\begin{equation}\label{concent}
\mu^{*n}*\delta_{x}(B_{\rho}y) \geq \rho^\kappa.
\end{equation}

 Let $\alpha = \frac{1}{10(\Lyap+1)} > 0$ and then $m = \lfloor \alpha \abs{\log \rho} \rfloor$.
Writing {$\mu^{*n}*\delta_{x}=\mu^{*m}*\mu^{*(n-m)}*\delta_{x}$}, Equation \eqref{concent} implies that 
\[\mu^{*(n-m)}*\delta_{x}(Z)\geq \rho^{2\kappa}\,\, \text{ where } \,\,Z:=\{z: \mu^{*m}*\delta_z(B_{\rho}y)\geq \rho^{2\kappa}\},\]
 up to assuming $\rho$ small enough in terms of $\kappa$. 
Indeed,
\begin{align*}
    \rho^{\kappa}\leq \mu^{*m}*\mu^{*(n-m)}*\delta_x(B_{\rho} y)&=\int_{Z\cup (X\smallsetminus Z)} \mu^{*m}*\delta_z(B_{\rho}y) \dd\mu^{*(n-m)}*\delta_x(z)\\
    &\leq \rho^{2\kappa}+\mu^{*(n-m)}*\delta_x(Z),
\end{align*}
so we obtain  $\mu^{*(n-m)}*\delta_x(Z)\geq \rho^{\kappa}-\rho^{2\kappa}\geq \rho^{2\kappa}$, provided $\rho\leq 2^{-1/\kappa}$.

We now show that $Z$ must be included in a small neighborhood of the cusp.
Fix $z\in Z$. By definition, 
\begin{equation} \label{posd-eq2}
{\mu^{*m}} \set{g \,:\, g z \in B_{\rho}y } \geq \rho^{2\kappa}.
\end{equation}
On the other hand, fixing $\gamma=\gamma(\mu) \in (0,1)$ as  in \Cref{posdim-sigma}, we have by \Cref{posdim-sigma}(i) that for   $\rho\lll_{\kappa} 1$,
\begin{equation} \label{eq-bg-bound}
\mu^{*m} \set{g \,:\, |\ttb_{g}|\leq \rho^{-4\gamma^{-1}\kappa} } \geq 1-\rho^{3\kappa }.
\end{equation}
By  the large deviation principle of i.i.d. random variables $(\ttr_g)_{g\sim \mu}$, there exists also $\eps>0$ depending only on $\mu$ such that
 \begin{equation}\label{eq-rg-bound}
\mu^{*m}\set{g\,:\,\log\ttr_{g} \in [-(\Lyap+1)m, -(\Lyap-1)m]}\geq 1-\rho^{\alpha\eps}.
\end{equation}

Let $C>1$ be a parameter to be specified below depending on $\mu$ only. Cutting the intervals $[-\rho^{-4\gamma^{-1}\kappa}, \rho^{-4\gamma^{-1}\kappa}]$ and $[-(\Lyap+1)m, -(\Lyap-1)m]$ into  subintervals of length  $\rho^{C\kappa}$, then using the pigeonhole principle, we deduce from \eqref{posd-eq2} \eqref{eq-bg-bound}, \eqref{eq-rg-bound} that there exists $(b_0,r_0)\in \R^2$ with  $\abs{b_{0}}\leq \rho^{-4\gamma^{-1}\kappa}$ and   $r_{0} \in[e^{-(\Lyap+1)m}, e^{-(\Lyap-1)m}]$ such that the set
\[E := \setbig{g \,:\, g z \in B_{\rho}y  \text{ and } \abs{\ttb_{g}-b_{0}}\leq \rho^{C\kappa} \text{ and } \abs{1-\ttr_{g}r^{-1}_{0}}\leq \rho^{C\kappa}}\]
has $\mu^{*m}$-measure
\begin{equation}\label{posd-eq5}
\mu^{*m}(E) \geq \frac{\rho^{2\kappa}-\rho^{3\kappa}-\rho^{\alpha\eps}}{\lceil2\rho^{-4\gamma^{-1}\kappa}\rho^{-C\kappa}\rceil\lceil 2 m\rho^{-C\kappa}\rceil} \geq \rho^{4C\kappa}
\end{equation}
where the last lower bound assumes $C\geq 4\gamma^{-1}$, $3\kappa\leq \alpha\eps$, and $\rho\lll_{\kappa}1$.

Consider $g_{1}, g_{2}\in E$.
By the bounds on $b_0$ and $r_0$ together with the choice of $m$, we have
$\norm{\Ad(g_1^{-1})} \leq \rho^{-1/2}$ provided $\kappa \lll_{C}1$.
Using $\dist(g_{1} z, g_{2} z) \ll \rho$, we deduce
\begin{equation}\label{eq:distzhz}
\dist(z, g_1^{-1}g_2 z) \ll \norm{\Ad(g_1^{-1})} \rho \ll \rho^{1/2}.
\end{equation}

We now aim to choose such $g_1$ and $g_2$ so that their mutual distance is much greater than $\rho^{1/2}$, but still dominated by a large power of $\rho^\kappa$.
Recalling that $g_i = a(\ttr_{g_{i}}^{-1})u(\ttb_{g_{i}})$ for each $i = 1,2$, we rewrite $g_{1}^{-1}g_{2}$ as
\begin{equation*}
g_{1}^{-1}g_{2}=u(-\ttb_{g_{2}}) h u(\ttb_{g_2}) \,\,\,\text{ where } \,\,\,h := u( \ttb_{g_2} - \ttb_{g_1} ) a(\ttr_{g_{1}}\ttr^{-1}_{g_{2}}).
\end{equation*}
The combination of  \eqref{posd-eq5} and the non-concentration estimate from \Cref{posdim-sigma}(ii) allows us to choose the elements $g_{1}, g_{2}\in E$ such that 
\begin{equation*}\label{posd-eq7}
\abs{\ttb_{g_2}-\ttb_{g_1}} \geq \rho^{\gamma^{-1}5C\kappa}
\end{equation*}
provided  $\kappa\lll_{C}1$ and $\rho\lll_{\kappa}1$  (in particular justifying $\rho^{\gamma^{-1}5C\kappa}> e^{-m}$ as required by \Cref{posdim-sigma}(ii)).
Observing that 
\begin{equation*}
\dist (h, \Id)
\simeq \abs{\ttb_{g_2}-\ttb_{g_{1}}|+|1-\ttr_{g_{1}}\ttr^{-1}_{g_{2}}}
\in [\rho^{\gamma^{-1}5C\kappa}, 4\rho^{C \kappa}]
\end{equation*}
and recalling $\abs{\ttb_{g_{2}}} \leq \rho^{-4\gamma^{-1} \kappa}$, 
we deduce
\begin{equation}\label{posd-eq6}
\rho^{1/4} \ll \rho^{\gamma^{-1}(5C + 8)\kappa} \ll \dist(g_1^{-1}g_2, \Id) \ll \rho^{(C - 8\gamma^{-1})\kappa}
\end{equation}
provided $\kappa\lll_{C} 1$. This is the desired separation for $g_{1},g_{2}$.

Assume $C >16 \gamma^{-1}$.
From \eqref{eq:distzhz} and \eqref{posd-eq6}, we deduce that $\inj(z) \ll \rho^{C \kappa/2} + \rho^{1/2}$.
When $\kappa \lll_C 1$ and $\rho \lll_\kappa 1$, this gives
\[
\inj(z) \leq \rho^{C \kappa/4}.
\]

In conclusion, we have shown that for $C\ggg1$, for $\kappa\lll_{C} 1$, $\rho\lll_{\kappa}1$, and $n \geq m =\lfloor \alpha|\log \rho| \rfloor$, we have
\[({\mu^{*(n-m)}}*\delta_{x} ) \{ \inj  \leq \rho^{C\kappa/4}\} \geq \rho^{2\kappa} \]
By the effective recurrence statement from \Cref{effective-recurrence}, this is absurd if $n-m\ggg {|\log \inj(x)|}$.
This concludes the proof of the proposition.
\end{proof}

\section{Dimensional bootstrap} \label{Sec-bootstrap}

In this section, we explain how the positive dimension estimate for $\mu^{*n}*\delta_{x}$ established in the previous section can be {upgraded} to a high-dimension estimate, up to applying more convolutions by $\mu$ and throwing away some small part of the measure. 
The notion of \emph{robust measures}  from \cite{Shmerkin} is well adapted to our purpose.
\begin{definition}[Robustness] \label{robust-measure}
Let $\alpha >0$, $I\subseteq (0, 1]$, $\tau\in \R^+$. 
A Borel measure $\nu$ on $X$ is \emph{$(\alpha, \cB_{I}, \tau)$-robust} if $\nu$ can be decomposed as the sum of two Borel measures $\nu=\nu'+\nu''$ such that $\nu''(X)\leq \tau$, and $\nu'$ satisfies 
\begin{equation} \label{eq:injrad}
 \nu'\{\inj < \sup I \}=0,
 \end{equation}
as well as for all $ \rho\in I$, $y\in X $,
\begin{equation}
\label{eq:rob1}
\nu'(B_ {\rho}y)\leq \rho^{3\alpha}.
\end{equation}
If $I$ is  a singleton $I = \{\rho\}$, we simply write that $\nu$ is $(\alpha, \cB_\rho, \tau)$-robust.
\end{definition}

Condition \eqref{eq:rob1} means that $\nu'$ has normalized dimension at least $\alpha$ with respect to balls of radius $\rho$. Note that in this definition, $\nu$ may not be a probability measure, this flexibility will be convenient for us.

The goal of the section is to establish the following high dimension estimate.

 \begin{proposition}[High dimension] \label{high-dim}
Let $\kappa \in (0, 1/10)$. For  $\eta, \rho \lll_{\kappa}1$ and  
 for all $n \ggg_{\kappa} |\log \rho| + |\log \inj(x)|$,  the measure {$\mu^{*n}*\delta_{x}$} is $(1-\kappa, \cB_ {\rho}, \rho^\eta)$-robust.
\end{proposition}

 \subsection{Multislicing}
The proof of \Cref{high-dim} relies on a multislicing estimate established in \cite{BH24}. We recall the case of interest in our context. 

\bigskip
We consider $\Theta$  a measurable space. We let $(\varphi_{\theta})_{\theta\in \Theta}$ denote a measurable family of $C^2$-embeddings
$\varphi_{\theta} : B^{\R^3}_{1}\rightarrow \R^3$, and $(L_{\theta})_{\theta\in \Theta}$  a measurable family of constants $L_{\theta}\geq 1$ such that  each map $\varphi_{\theta}$ is $L_{\theta}$-bi-Lipschitz:
\begin{align*}
\forall x,y\in B^{\R^3}_1,\quad  \frac{1}{L_{\theta}}\|x-y\|\leq \|\varphi_\theta(x) - \varphi_{\theta}(y)\| \leq L_{\theta}\|x-y\|
\end{align*}
and has  second order derivatives bounded by $L_{\theta}$:
\begin{align*}
\forall x, h \in B^{\R^3}_1,\quad \norm{\varphi_\theta(x + h) - \varphi_\theta(x) - (D_x \varphi_\theta)(h)} \leq L_{\theta}\norm{h}^2.
\end{align*}

Given $\rho >0$, we denote by $\cD_{\rho}$  (resp. $\cR_\rho$)  the collection of subsets of $\R^3$ that are translates of the $\rho$-cube $[0,\rho]^3$ (resp. the rectangle $R_{\rho}:= [0,1]e_{1}+[0,\rho^{1/2}]e_{2}+[0,\rho]e_{3}$).

We will also need to measure the angle between subspaces in $\R^3$.
For each $k = 1,2,3$, endow $\wedge^k \R^3$ with the unique Euclidean structure with respect to which the standard basis is orthonormal.
Given subspaces $V, W\subseteq \R^3$, we set
\[\dang(V,W)=\| v \wedge w\|\]
where $v,w$ are  unit vectors in $\wedge^{*}\R^3$ spanning respectively the lines $\wedge^{\dim V}V$, $\wedge^{\dim W}W$.

\bigskip

The multislicing estimate presented in \Cref{multislicing} below is a special case of \cite[Corollary 2.2]{BH24}.
It takes as input  a Borel measure $\nu$ on the unit ball $B^{\R^3}_{1}$ that has normalized dimension at least $\alpha$ with respect to balls of radius above $\rho$.  The output is a dimensional gain when the balls are replaced by    (non-linear) rectangles of the form $(\varphi^{-1}_{\theta}(x+R_{\rho}))_{x\in \R^3}$ provided $\theta$ is chosen {almost} typically via a probability measure for which  $\varphi_{\theta}$ satisfies suitable bounds on the derivatives as well as non-concentration estimates. The proof relies on Shmerkin's nonlinear version~\cite{Shmerkin} of Bourgain's discretized projection theorem~\cite{Bourgain2010}, local conditioning arguments, and a submodular inequality for covering numbers.

\begin{thm}[Multislicing \cite{BH24}] \label{multislicing}
Given $\kappa \in (0, 1/2)$, there exist $\eps=\eps(\kappa)>0$ and $\rho_{0}=\rho_{0}(\kappa)>0$ such that the following holds for all $\rho\in (0, \rho_{0}]$.

{Let} $\nu$ be a Borel  measure on $B^{\R^3}_{1}$ satisfying:  $\exists \alpha\in  (\kappa, 1-\kappa)$, $\forall r \in [\rho, \rho^\eps]$, 
\[ \sup_{Q\in \cD_{r}} \nu(Q)\leq  r^{3\alpha}.\]

Let $\Xi$ be a probability measure on $\Theta$ satisfying: 
\begin{itemize}
\item[(i)] \[\Xi \set{\theta\in\Theta : L_{\theta}\leq \rho^{-\eps}} =1.\]

\item[(ii)] $\forall k\in \{1,2\}$, $\forall x\in B^{\R^3}_{1}$, $\forall r \in [\rho, \rho^\eps]$,  $\forall W \in \Gr(\R^3, 3 - k)$,
\[
\Xi \set{\theta\in\Theta : \dang((D_x \varphi_{\theta})^{-1}V_k, W)\leq r }\leq r^\kappa,
\]
where $V_k = \Span_{\R}(e_1,\dotsc,e_k)$.
\end{itemize}

Then there exists $\cF\subseteq \Theta$ such that $\Xi(\cF)\geq 1- \rho^\eps$ and for every $\theta \in \cF$, there exists $A_{\theta}\subseteq B^{\R^3}_{1}$ with $\nu(A_{\theta})\geq 1-\rho^\eps$ and satisfying 
\[
\sup_{Q\in \cR_{\rho}}\nu_{|A_{\theta}}(\varphi_{\theta}^{-1}Q) \leq \rho^{\frac{3}{2}\alpha+\eps}.
\]
\end{thm}


\subsection{Straightening charts}

In order to apply the multislicing estimates from  \Cref{multislicing}, we need special macroscopic charts in which the preimage by {$g\in \supp \mu^{*n}$} of a  ball looks like a rectangle. The goal of the present subsection is to define those charts.

\bigskip
Recall that $\kg$ admits the rootspace decomposition 
\[
\kg= \kg_-\oplus \kg_{0}\oplus \kg_+,
\]
where
\[
\kg_- 
=\R e_{-},\quad  \kg_{0} 
=\R e_{0}\quad \text{and} \quad \kg_{+}=\R e_{+}.
\]
We then define $\Psi : \kg\rightarrow  G$ by the formula: $\forall (v_-, v_0, v_+)  \in\kg_- \times {\kg_0} \times  \kg_+$, 
\[
\Psi(v_- + v_0 + v_+) = \exp(v_-)\exp(v_0)\exp(v_+).
\]

Recall also the notation 
\[
a(t)= \begin{pmatrix} t^{1/2} &0 \\  0& t^{-1/2}\end{pmatrix}.
\]

The next lemma tells us that, in the chart $\Psi$,  the image of a  ball $B_{\rho}$ in $ G$ by some diagonal element $a(t)$ with small $t>0$ is included in a  rectangle whose volume is comparable.

\begin{lemma} \label{straightening}
\label{lm:param-loc} 
There is an absolute constant $r_0> 0$ such that 
for any $t, \rho\in (0, 1)$  with $|t^{-1}\rho|\leq r_0$ and any $h \in G$, there is $w \in \kg$ such that
\begin{equation}\label{eq:param-loc}
\set{v \in B^\kg_{r_0} : \Psi(v) \in a(t) B_{\rho} h} \subset \Ad(a(t)) B^{\kg}_{10\rho} + w.
\end{equation}
\end{lemma}

This result is a particular case of \cite[Lemma 4.10]{BH24}. We give a shorter proof in our context for completeness.

\begin{proof} 
Fix a vector $w$ in the left hand side of \eqref{eq:param-loc}. If $v$ belongs to the left hand side of \eqref{eq:param-loc} as well, then by the triangle inequality, we have
    \[\Psi(v)\in a(t)B_{2\rho} a(t^{-1})\Psi(w).\]
We can choose $r_0>0$ small so that the image of $\Psi$ contains $B_{2r_{0}}$. In particular, using that conjugation commutes with the exponential map, there is $u=u_-+u_0+u_+\in \mathfrak{g}$ such that 
\[u_-\in B^{\mathfrak{g}_-}_{t^{-1}\rho}, \quad u_0\in B^{\kg_{0}}_{\rho},\quad u_+\in B^{\mathfrak{g}_+}_{t\rho}\]
and 
\[\Psi(v)=\Psi(u)\Psi(w).\]
 Consider   $x,y,s\in \R$ with $s\neq 0$ and $1+xy\neq 0$.  Note that we have in $G$ the following equality
\begin{align} \label{eqmatrix-1}
 \begin{pmatrix}
        1 & x\\
        0 & 1
    \end{pmatrix}
    \begin{pmatrix}
        1 & 0\\
        y & 1
    \end{pmatrix}
    \begin{pmatrix}
        s & 0\\
        0 & s^{-1}
    \end{pmatrix}
=
    \begin{pmatrix}
        1 & 0\\
        y' & 1
    \end{pmatrix}
    \begin{pmatrix}
        s' & 0\\
        0 & s'^{-1}
    \end{pmatrix}
    \begin{pmatrix}
        1 & x'\\
        0 & 1
    \end{pmatrix}
 \end{align}
where
\[x'= \frac{x}{(1+xy)s^2}, \quad s'= (1+xy)s,\quad y'= \frac{y}{1+xy}.\]
Similarly, 
\begin{align} \label{eqmatrix-2}
    \begin{pmatrix}
        s & 0\\
        0 & s^{-1}
    \end{pmatrix}
    \begin{pmatrix}
        1 & 0\\
        y & 1
    \end{pmatrix}=
    \begin{pmatrix}
        1 & 0\\
        s^{-2} y & 1
    \end{pmatrix}
    \begin{pmatrix}
        s & 0\\
        0 & s^{-1}
    \end{pmatrix}.
\end{align}
Observe that 
\[\Psi(v)=\Psi(u)\Psi(w)=\exp(u_-)\exp(u_0)\exp(u_+)\exp(w_-)\exp(w_0)\exp(w_+).\]
Assuming $r_0\lll 1$,  applying \eqref{eqmatrix-1} to the factor $\exp(u_+)\exp(w_-)\exp(w_0)$
then \eqref{eqmatrix-2} to the factor $\exp(u_0)\exp(w'_-)$,  we obtain 
\[\Psi(v)\in \exp(B^{\mathfrak{g}_-}_{10t^{-1}\rho}+w_-)\exp(B^{\mathfrak{g}_0}_{10\rho}+w_0)\exp(B^{\mathfrak{g}_+}_{10t \rho}+w_+).\]
Noting that $\Psi$ is injective (by direct computation again), this finishes the proof. 
\end{proof}

In view of \Cref{straightening} and the formula
\[
g = a(\ttr^{-1}_{g})u(\ttb_{g}),
\]
we define a family of \emph{straightening charts} $(\varphi_{\theta})$ as follows. Let 
\[\Theta=u(\R).\]
Fix $r_{1}>0$ such that $\Psi$ is a smooth diffeomorphism between $B^{\kg}_{r_{1}}$ and a neighborhood  $\mathcal{O}$ of $\Id\in  G$.
Given $\theta\in \Theta$, define $\varphi_{\theta}: \mathcal{O} \rightarrow \kg$ by
\[\varphi_{\theta}:= \Ad(\theta^{-1})\circ (\Psi_{|B^{\kg}_{r_{1}}})^{-1}.\]
Using that $\Psi$ commutes with conjugation, we have the alternative formula $\varphi_{\theta}=(\Psi_{|\Ad(\theta^{-1})B^{\kg}_{r_{1}}})^{-1} \circ \sC_{\theta^{-1}}$ where $\sC_{\theta^{-1}} : h\mapsto \theta^{-1}h\theta$.
Note that $\varphi_{\theta}$ is $L_{\theta}$-{bi-Lipschitz} and satisfies  $\|\varphi_{\theta}\|_{C^2}\leq L_{\theta}$ for some quantity
\begin{equation}
\label{eq:Ltheta}
L_{\theta} := L \|\theta\|^{4}.
\end{equation}
where $L > 1$ is a constant depending only on $r_1$.

Given an element $g\in P$, write  
\[g^{-1}= \theta_{g} a(\ttr_g)\] 
with 
\begin{equation}
\label{eq:thetag}
\theta_{g} := u(-\ttb_g) \in \Theta.
\end{equation}
\Cref{straightening} tells us that for any $h\in G$, $\varphi_{\theta_{g}}(g^{-1} B_{\rho}h)$  is essentially an additive translate of the rectangle $\Ad(a(\ttr_{g}))B^\kg_{\rho}$, provided that $\ttr_{g}\in (0, 1)$ and that both $g^{-1} B_{\rho}h$ and $a(\ttr_{g}) B_{\rho}h\theta_{g}$ sit inside a prescribed (macroscopic) neighborhood of the identity.

\subsection{Control of the charts} \label{Sec-non-concentration}

We now check that the charts {($\varphi_{\theta}$)} from the previous section satisfy {distortion control} and non-concentration estimates. {These} will be required in order to apply  \Cref{multislicing} in the next section. The constants $r_{0}$ from \Cref{straightening} and $r_{1}>0$ in the definition of $\varphi_{\theta}$ are assumed fixed in a canonical way (so that dependence on them does not appear in subscript of asymptotic notations).

{Recall $L_\theta$ and $\theta_g$ are respectively defined in \eqref{eq:Ltheta}  and \eqref{eq:thetag}.}

\begin{lemma}[Distortion control]\label{lm:nondeg}
Given $\eps>0$, there exists $\gamma=\gamma(\mu, \eps)>0$ such that for $n\ggg_{\eps} 1$, we have
\[
\mu^{*n}\set{ g \,: \, L_{\theta_{g}}> e^{\eps n}} \leq e^{-\gamma n}.
\]
\end{lemma}
\begin{proof}
This is a direct consequence of $L_{\theta_{g}}\ll\|\theta_{g}\|^{4} \ll(1+|\ttb_{g}|)^{4}$ and \Cref{posdim-sigma} (i), stating that the variable $\ttb_{g}$, where $g\overset{law}{\sim} \mu^{*n}$, has a moment of positive order that is bounded independently of $n$. 
\end{proof}

Set $\kg_{-,0}=\kg_{-} \oplus \kg_{0}$. 
\begin{lemma}[Non-concentration]\label{lm:noncon}
There exists a constant $\kappa > 0$ such that for $n \ggg 1$,  $h \in B_{r_{0}}$ and $\rho \geq e^{-n}$, we have
\[
\forall W \in \Gr(T_h G, 2), \quad \mu^{*n}\set{ g \, :\, \dang( (D_h \varphi_{\theta_{g}})^{-1}\kg_{-} ,W) \leq \rho} \ll \rho^\kappa,
\]
and 
\[
\forall W \in \Gr(T_h G, 1),\quad \mu^{*n}\set{ g \, :\, \dang((D_h \varphi_{\theta_{g}})^{-1}\kg_{-,0},W) \leq \rho} \ll \rho^\kappa.
\]
\end{lemma}

\begin{proof}
Unwrapping definitions, we observe that the distribution of subspaces $h\mapsto (D_h \varphi_{\theta})^{-1}\kg_{-}$ is right-invariant and coincides with $\Ad(\theta)\kg_{-}$ at the identity. The same holds for $(D_h \varphi_{\theta})^{-1}\kg_{-,0}$. Recalling that  $\theta_{g}=u(-\ttb_{g})$ and $\sigma^{(n)}$ is the law of $\ttb_{g}$ as $g \overset{law}{\sim} \mu^{*n}$, we are then led to proving the following non-concentration estimates:
\[ \sup_{W\in \Gr(\kg, 2)}\sigma^{(n)} \{ s \, :\,  \dang( \Ad(u(-s)) \kg_{-}, W)\leq \rho\} \ll \rho^\kappa, \quad \text{and}\]
\[ \sup_{W\in \Gr(\kg, 1)} \sigma^{(n)}\{ s \, :\, \dang( \Ad(u(-s)) \kg_{-,0}, W)\leq \rho\} \ll \rho^\kappa.\]

Let us check the first estimate, where $W\in \Gr(\kg, 2)$. Set $e_{-,0}=e_{-}\wedge e_{0}$,  $e_{-,+}=e_{-}\wedge e_{+}$,  $e_{0,+}=e_{0}\wedge e_{+}$. 
Write $\wedge^2 W= \R(a e_{-,0}+be_{-,+}+ce_{0,+})$ where $a,b,c\in \R$ satisfy $\max(|a|,|b|,|c|)=1$. Then direct computation yields for any $s\in \R$, 
\[\dang(\Ad(u(-s)) \kg_{-}, W) \simeq \frac{\abs{a s^2 -bs-c}}{s^2+|s|+1}.\] 
Hence $\dang(\Ad(u(-s)) \kg_{-}, W) \leq \rho$ implies either $\abs{s} \geq \rho^{-1/3}$ or $\abs{a s^2 - b s - c} \ll \rho^{1/3}$.
Applying respectively \Cref{posdim-sigma} and \Cref{computing-non-concentration}, we obtain the desired non-concentration.

The second estimate is similar: writing $W= \R(a e_{+}+ b e_{0}+ c e_{-})$ where $\max(\abs{a},\abs{b},\abs{c})=1$, we find 
$\dang( \Ad(u(-s)) \kg_{-,0}, W) \simeq \frac{\abs{a - 2 b s - c s^2}}{1+|s|+s^2}$.
\end{proof}

\subsection{Dimension increment}

In this subsection, we apply the multislicing estimate from \Cref{multislicing} to show that convolution by a well chosen power of $\mu$ increases dimensional properties of a measure at a given scale.

\begin{proposition}[Dimension increment] \label{dim-increment}
Let $\kappa, \eps, \rho \in (0,1/10)$, $\alpha \in {[\kappa, 1-\kappa]}$, $\tau \geq 0$ be some parameters. 
Consider on $X$ a Borel measure $\nu$  which is $(\alpha,   \cB_{[\rho, \rho^{\eps}]}, \tau)$-robust. Denote by $n_{\rho} \geq 0$  the integer part of $\frac{1}{2 \Lyap}|\log \rho|$.

Assume $\eps,  \rho \lll_{ \kappa} 1$,
then
$$
\text{$\mu^{*n_{\rho}}*\nu$ is  $(\alpha+\eps, \cB_{\rho^{1/2}}, \tau+\rho^{\eps})$-robust}.
$$
\end{proposition}

\noindent\emph{Remark}.
Recall here that $\Lyap$ denotes the Lyapunov exponent of the $\Ad_{\star}\mu$-walk on $\kg$. Hence our choice for $n_{\rho}$ guarantees that  the operator norm of $\Ad g^{-1}$ is roughly $\rho^{-1/2}$ when $g\overset{law}{\sim} {\mu^{*n_{\rho}}}$.

\begin{proof}
In the proof, we may allow $\rho$ to be small enough in terms of $\eps$ (not just $\Lambda, \mu, \kappa$). We may also assume $\tau=0$. We will write $n=n_{\rho}$, and $\|\cdot\|$ the total variation norm on signed measures.

Note that the compact set $X_{\rho^\eps}:=\{\inj \geq \rho^\eps\}$ can be covered by $\rho^{-O(\eps)}$ balls of radius $\rho^{2\eps}$, more precisely 
\[X_{\rho^\eps} \subseteq \cup_{i\in I} B_{\rho^{2 \eps}}x_{i}\]
where $\sharp I\leq \rho^{-O(\eps)}$, $x_{i}\in X_{\rho^\eps}$ for all $i$. As $\nu$ is supported on $X_{\rho^\eps}$, we can then write 
\[\nu=\sum_{i\in I} \nu_{i}*\delta_{x_{i}}\]
where  $\nu_{i}$ is a Borel measure on $ G$ with support in $B_{\rho^{2 \eps}}$.
Note that  the assumption that $\nu$ is $(\alpha,  \cB_{[\rho, \rho^{\eps}]}, 0)$-robust implies that each $\nu_{i} * \delta_{x_i}$ is $(\alpha, \cB_{[\rho, \rho^{\eps}]}, 0)$-robust. 
It follows that for each $i \in I$, $\nu_i$ satisfies the non-concentration property
\[
\forall r \in [\rho, \rho^{\eps}],\quad \sup_{h \in G} \nu_i(B_r h) \leq r^{3\alpha}.
\]

We now apply \Cref{multislicing} to each $\nu_{i}$. We consider the family of charts $\varphi_{\theta} : \mathcal{U}\rightarrow \kg$ introduced in \Cref{Sec-non-concentration}. In order to guarantee the {distortion control} requirement for $\varphi_{\theta}$, we introduce the renormalized truncation of $\mu^{*n}$ defined by
\[\mu'_{n}= \frac{\mu^{*n}_{|L_{\theta_{g}}\leq \rho^{-\eps}}}{\mu^{*n}\{L_{\theta_{g}}\leq \rho^{-\eps}\}}. \]
By \Cref{lm:nondeg}, this probability measure satisfies $\|\mu'_{n}-\mu^{*n} \| \leq \rho^{\gamma}$ for some $\gamma=\gamma(\mu, \eps)>0$. In particular, provided $\rho \lll_{\eps}1$, the measure $\mu'_{n}$ also satisfies the non-concentration estimates  from \Cref{lm:noncon}. This allows us to apply  \Cref{multislicing} with 
 $\Xi$ the law of $\theta_{g}$ when $g \overset{law}{\sim} {\mu'_{n}}$ ($n=n_{\rho}$). We obtain some constant  $\eps_{1}>0$ depending only on $\kappa$, $\mu$ such that up to assuming $\eps \lll_{\kappa}1$, $\rho\lll_{\kappa, \eps}1$, there exists $\G_{i}\subseteq P$ with ${\mu'_{n}}(\G_{i})\geq 1-\rho^{\eps_{1}}$ satisfying for every $g\in \G_{i}$, that there exists a Borel measure $\nu_{i,g}\leq \nu_{i}$ with $\nu_{i,g}(G)\geq \nu_{i}(G)-\rho^{\eps_{1}}$ and such that 
\begin{equation} \label{maj-nuig}
\sup_{Q\in \cR_{\rho} }\nu_{i,g}(\varphi_{\theta_{g}}^{-1}Q) \leq \rho^{\frac{3}{2}\alpha+\eps_{1}}. 
\end{equation}

On the other hand, the large deviation principle for the walk on $\R$ driven by $-\log \ttr_g \dd \mu(g)$ guarantees that
\[\text{the set $\G_{\ttr}=\{\, g\,:\,{\ttr_g^{-1}} \in [\rho^{-1/2+\eps}, \rho^{-1/2-\eps}] \,\}\,\,$ satisfies $\,\,{\mu'_{n}(\G_{\ttr})} \geq 1- \rho^{\eps_{2}}$}\]
for some $\eps_{2}=\eps_{2}(\mu, \eps)>0$. 

Setting $\G_{i,\ttr}=\G_{i} \cap \G_{\ttr}$ and using \Cref{straightening}, observe that for $i\in I$, $g\in \G_{i,\ttr}$, for any ball $B_{\rho^{1/2}}y$  where $y\in X$, the intersection  $(g^{-1}B_{\rho^{1/2}}y) \cap B_{\rho^{2\eps}}x_{i}$ lifted to $B_{\rho^{2\eps}}$ is included in at most $\rho^{-O(\eps)}$ blocks of the form $\varphi_{\theta_{g}}^{-1}Q$ where $Q\in \cR_{\rho}$.  Hence, we get from \eqref{maj-nuig},
\begin{equation} \label{gnuig}
\sup_{y\in X} \delta_{g}*\nu_{i,g}(B_{\rho^{1/2}}y) \leq \rho^{\frac{3}{2}\alpha+\eps_{1}-O(\eps)}. 
\end{equation}

Setting $\G=\cap_{i\in I} \G_{i,\ttr}$ and recalling $\sharp I\leq \rho^{-O(\eps)}$, we have $\mu'_{n}(\G)\geq 1-\rho^{\eps_{1}-O(\eps)}-\rho^{\eps_{2}}$. We deduce 
\[\mu^{*n}( \{L_{\theta_{g}}\leq \rho^{-\eps} \}\cap \G) \geq 1-\rho^{\eps_{1}-O(\eps)}-\rho^{\eps_{2}}-\rho^{\gamma}.\]
  Letting 
\[m_{n}''={\mu^{*n}}* \nu- \int_{\{L_{\theta_{g}}\leq \rho^{-\eps} \}\cap \G} \sum_{i}\delta_{g}*\nu_{i,g} \dd\mu^{*n}(g),\]
and taking $\eps\lll_{\eps_{1}}1$, we have $\|m_{n}''\|\leq \rho^{\eps_{3}}$ where $\eps_{3}=\eps_{3}(\eps, \eps_{1}, \eps_{2}, \gamma)>0$, while we see from \eqref{gnuig}  that $m_{n}':=({\mu^{*n}}* \nu)-m_{n}''$  satisfies 
\[\sup_{y\in X} m'_{n}(B_{\rho^{1/2}}y) \leq \rho^{\frac{3}{2}\alpha+\eps_{1}/2}. \]

We now have checked  the required dimensional increment. 
In order to conclude, we also need to check that $\mu^{*n}*\nu$ does not give too much mass to the cusp. 
Indeed, \Cref{effective-recurrence} implies that for some constants $c,c'>0$ depending on $\mu$,  we have for all $x \in X$,
\[\mu^{*n}*\delta_{x}\{\inj < \rho^{1/2}\,\}\ll ({\inj^{-c}(x)} e^{-c'n} +1)\rho^{c/2}. \]
Integrating over $x$ with respect to $\nu$, imposing $\eps<c'/(2\Lyap c)$, and recalling that $\nu$ is supported on $\{\inj \geq \rho^\eps\}$ by assumption while $n=n_{\rho}$, we obtain $\mu^{*n} * \nu \{\inj < \rho^{1/2}\} \ll \rho^{c/2}$. 
\end{proof}

\subsection{Proof of high dimension}

We are finally able to show \Cref{high-dim}, namely that ${\mu^{*n}}*\delta_{x}$ reaches high dimension  exponentially fast. The proof starts from positive dimension given by \Cref{positive-dimension} and then proceeds by small increments using \Cref{dim-increment}. 
Note however that  \Cref{dim-increment} assumes non-concentration on a wide range of scales but the output dimensional increment only concerns a specific scale. Hence we need to combine those single-scale increments to allow iterating the bootstrap. For this, we rely on the following lemma.

\begin{lemma}\label{Lemma: combinatorial robust 2} 
    Let $\alpha,s,\rho\in (0,1]$, $\tau\in \mathbb{R}^+$ be parameters. If $\nu$ is $(\alpha,\mathcal{B}_{r},\tau)$-robust for all $r\in [\rho,\rho^s]$, then for any $\eps\in (0,\alpha)$, the measure $\nu$ is $(\alpha-\eps,\mathcal{B}_{[\rho,\rho^s]},\lceil \frac{\log s}{\log(1-\eps)}\rceil \tau)$-robust.
\end{lemma}

\begin{proof}
This is just a combination of {two} observations (1) 
if $\nu$ is $(\alpha,\mathcal{B}_{r},\tau)$-robust, then for every $t\in (0,1)$, it is $(t\alpha, \mathcal{B}_{[r^{1/t},r]},\tau)$-robust; (2) if $\nu$ is $(\alpha,\mathcal{B}_{I_1},\tau_1)$-robust and $(\alpha,\mathcal{B}_{I_2},\tau_2)$-robust, then $\nu$ is $(\alpha,\mathcal{B}_{I_1\cup I_2},\tau_1+\tau_2)$-robust. See \cite[Lemma 4.5]{BH24} for details. 
\end{proof}

\begin{proof}[Proof of \Cref{high-dim}] 
Let $A>0$ be a large enough constant depending on the initial data $\Lambda, \mu$. Combining \Cref{positive-dimension} and \Cref{effective-recurrence}, we may assume $\kappa>0$ small enough from the start, so that for any $M>0$, for every $\rho\lll_{M}1$ and $n\geq M |\log \rho|+ A|\log \inj(x)|$,  the measure
\[\text{{$\mu^{*n}*\delta_{x}$} is $(\kappa, \cB_{[\rho^M, \rho^{1/M}]}, \rho^{\kappa/M})$-robust}. \]

Let $\eps_0, \rho_0 \in (0, 1/2)$ be constants depending only on $\kappa$ such that  the conclusion of \Cref{dim-increment} holds for all $\alpha \in [\kappa, 1 - \kappa]$, $\eps \leq \eps_0$, and $\rho \leq \rho_0$.
Fix $\eps = \eps_0 /2$.
Let $K = \left\lfloor \frac{1 - 2 \kappa}{\eps}\right\rfloor + 1$ and then
$M = \eps^{-K}$.
Finally, let $\rho \leq \rho_0^M$ with $\rho\lll_{M}1$ as in the first paragraph.
We show by  induction that for every integer $0 \leq k \leq K $,
\begin{equation}\label{eq:ih-incre}
\begin{split}
\forall n \geq \,& t_k  := \left(1 + \frac{k}{2 \Lyap}\right)M \abs{\log \rho} + A \abs{\log \inj(x)},\\
& \mu^{* n} * \delta_x \text{ is } \bigl(\kappa + k \eps, \cB_{\bigl[\rho^{M/2^k},\, \rho^{1/(2^k \eps^k M)}\bigr]}, O_{\kappa,k}(\rho^{\kappa/M})\bigr)\text{-robust.}
\end{split}
\end{equation}
Taking   $k=K$ in \eqref{eq:ih-incre},  we obtain \Cref{high-dim} since $\kappa + K \eps \geq 1 - \kappa$ and the interval $[\rho^{M/2^K}, \rho^{1/(2^K \eps^K M)}\bigr]$ contains $\rho$.

It remains to show \eqref{eq:ih-incre}  by induction on $k$.
The base case $k=0$ is given by the discussion in the first paragraph.
We now assume that \eqref{eq:ih-incre} holds for some $k < K $, and we prove it for $k + 1$.

Let $n \geq t_{k+1}$.
For every $r \in [\rho^{M/2^k}, \rho^{1/(2^k \eps^{k+1}M)}]$, write 
$n = \lfloor\frac{1}{2\Lyap} \abs{\log r}\rfloor + n'$ where $n' = n - \lfloor\frac{1}{2\Lyap} \abs{\log r}\rfloor \geq t_k$.
Apply \Cref{dim-increment} to the scale $r$ and the measure $\mu^{*n'} * \delta_x$ which we know from \eqref{eq:ih-incre} is $(\kappa + k \eps, \cB_{[r,r^\eps]}, O_{\kappa,k}(\rho^{\kappa/M}))$-robust.
We obtain that $\mu^{*n}*\delta_x$ is $(\kappa + (k + 2)\eps, \cB_{r^{1/2}}, O_{\kappa,k}(\rho^{\kappa/M}) + r^\eps)$-robust.
This being true for all $r \in [\rho^{M/2^k}, \rho^{1/(2^k \eps^{k+1}M)}]$, we can use \Cref{Lemma: combinatorial robust 2} to conclude the proof of the induction step.
\end{proof}

\section{From high dimension to equidistribution} \label{Sec-equidistribution}

We consider the one-parameter family of probability measures $(\eta_{t})_{t>0}$ on $G$ given  by 
\[\eta_t=a(t)u(s) \dd\sigma(s).\]
We show \Cref{endgame}, stating  that as $t\to +\infty$, a  probability measure on $X$ with dimension close to $3$ equidistributes under the $\eta_{t}$-process toward the Haar measure on $X$, and does so with exponential rate.
From this we deduce \Cref{ThB'} (whence \ref{Kintchine-dynamics}) and \Cref{ThC'}  (whence \ref{mu^n-equidistribution}).

\begin{proposition} \label{endgame}
There exist $\kappa, \rho_{0}>0$ such that the following holds for all $\rho \in (0,\rho_{0}]$, $\tau\in \R^+$. 

Let $\nu$ be a Borel  measure on $X$ that is $(1-\kappa, \cB_{\rho}, \tau)$-robust and has mass at most $1$. Then for any $t \in [\rho^{-1/4},\rho^{-1/2}]$, for any $f \in B^{\infty}_{\infty,1}(X)$ with $m_{X}(f)=0$,  we have 
\[
\abs{\eta_{t}*\nu(f)} \leq  (\rho^\kappa +\tau) \cS_{\infty,1}(f).
\]
\end{proposition}

The argument relies on the quantitative decay of correlations for $X$.
Consider the unitary representation of $G$ on $L^2(X)$ defined by the formula $g.f = f \circ g^{-1}$.
From the combination of  \cite[Lemma 3]{Bekka98} and \cite[Equations (6.1), (6.9)]{EMV09}, we know there exists  $\delta_{0}=\delta_{0}(\Lambda)>0$ such that for any function $f \in B^{\infty}_{2,1}(X)$ with $m_X(f) = 0$, any $g \in G$, we have
\begin{equation}
\label{eq:decay cor}
 \abs{\langle g.f, f \rangle} \ll \norm{g}^{-\delta_{0}} \cS_{2,1}(f)^2.
\end{equation}

From this we deduce a spectral gap for the family of Markov operators $P_{\eta_{t}}$. Recall that $P_{\eta_{t}}$ is the operator acting on non-negative measurable functions on $X$ given by the formula
\[P_{\eta_{t}} f (x)= \int_{G} f(gx) \dd\eta_{t}(g).\]
$P_{\eta_{t}}$ extends continuously into an operator on $ L^2(X)$ of norm $1$. 

\begin{proposition}[Spectral gap for $P_{\eta_{t}}$] \label{spectralgap-mu}
There exists $c>0$ such that for any function $f \in B^{\infty}_{2,1}(X)$ with $m_{X}(f)=0$, we have 
\[\forall t > 1, \quad  \|P_{\eta_{t}} f\|_{L^2} \ll  t^{-c} \cS_{2,1}(f).\]
\end{proposition}

\begin{proof}
Using \eqref{eq:decay cor}, we have 
\begin{align*}
\|P_{\eta_{t}} f\|^2_{L^2} 
&= \iint_{ G^2}\langle g^{-1} .f, h^{-1}.f\rangle \dd\eta_{t}(g) \dd\eta_{t}(h)\\
&= \iint_{ G^2}\langle hg^{-1} .f, f\rangle \dd\eta_{t}(g) \dd\eta_{t}(h)\\
&\ll \cS_{2,1}(f)^2 \iint_{ G^2} \norm{hg^{-1}}^{-\delta_{0}} \dd\eta_{t}(g) \dd\eta_{t}(h)
\end{align*}
Plugging in the definition of $\eta_t$ 
\begin{align*}
    hg^{-1} \dd\eta_t(g)\dd\eta_t(h)=u(t(s_1-s_2)) \dd\sigma(s_1) \dd\sigma(s_2),
\end{align*}
we get
\[
\norm{P_{\eta_t}f}^2_{L^2} 
\ll \iint_{\R^2} \max \{1 , t |s_1-s_2| \}^{-\delta_{0}} \dd\sigma(s_1) \dd\sigma(s_2) \cS_{2,1}(f)^2.
\]
Finally, the H\"older regularity of $\sigma$ from \Cref{sigma Holder}(ii) implies that
\[\sigma^{\otimes 2}\underbrace{\{(s_1,s_2) : t |s_1-s_2|\leq t^{1/2}\}}_E \ll t^{-c},\]
for some constant $c=c(\sigma) > 0$. Hence,
\begin{align*}
&\iint_{\R^2}\max \{1 , t \abs{s_1-s_2} \}^{-\delta_{0}} \dd\sigma(s_1) \dd\sigma(s_2)\\
= &\iint_{E} 1 \dd\sigma \otimes \sigma + \iint_{\R^2\smallsetminus E} (t \abs{s_1-s_2})^{-\delta_{0}} \dd\sigma(s_1) \dd\sigma(s_2) \\
\ll & t^{-c}+t^{-\delta_{0}/2},
\end{align*}
concluding the proof. 
\end{proof}

To prove \Cref{endgame} we \emph{mollify} the measure $\nu$ at some scale $\rho> 0$: 
let $\nu_\rho$ be the Borel measure on $X$ defined by
\[\nu_{\rho}(f)= \frac{1}{m_G(B_{\rho})} \int_X \int_{B_\rho} f(gx) \dd m_{G}(g) \dd\nu(x),\]
where 
$f$ denotes here any non-negative measurable function on $X$. 

Note that if $\nu$ is supported on the compact set $\{\inj \geq  \rho\}$, then by a change of variable $g \in B_\rho \mapsto gx \in B_\rho x$ and  the Fubini-Lebesgue theorem, we have 
\begin{align*}
    \nu_{\rho}(f)&= \frac{1}{m_G(B_{\rho})} \iint_{X \times X} \1_{y \in B_\rho x} f(y) \dd m_X(y) \dd\nu(x)\\
    &= \frac{1}{m_G(B_{\rho})} \int_{X} f(y) \int_{X} \1_{x \in B_\rho y} \dd\nu(x) \dd m_X(y).
\end{align*}
This implies that $\nu$ is absolutely continuous with respect to $m_X$ and its Radon-Nikodym derivative is
\begin{equation} \label{formula-nud}
\frac{\mathrm{d}\nu_{\rho}}{\mathrm{d} m_{X}}(x)= \frac{\nu(B_\rho x)}{m_G(B_{\rho})}.
\end{equation}
In particular, if $\nu$ is $(1-\kappa, \cB_{\rho}, 0)$-robust, then 
\[
\normBig{ \frac{\mathrm{d}\nu_\rho}{\mathrm{d} m_X}}_{L^\infty} \ll \rho^{-3\kappa}.
\] 

\begin{proof}[Proof of \Cref{endgame}]
We let $\kappa, r, \rho_{0}>0$ be parameters to specify below,  $\rho, \tau, \nu$ as in the proposition, and consider a test function $f\in B^\infty_{\infty,1}(X)$ with zero average.  
Clearly we may assume $\tau=0$, i.e. $\nu$ is $(1-\kappa, \cB_{\rho}, 0)$-robust.

We can write for any  $t>1$, 
\begin{align*}
| \eta_{t}*\nu(f)|= \abse{\int_{X} P_{\eta_{t}} f \dd \nu} 
\leq \abse{\int_{X} P_{\eta_{t}} f \dd \nu_\rho}  + \abse{\int_{X} P_{\eta_{t}} f \dd \nu_\rho - \int_{X} P_{\eta_{t}} f \dd \nu}.
\end{align*}
The first term is bounded  by
\begin{align*}
\abse{\int_{X} P_{\eta_{t}} f \dd \nu_\rho} & \leq \norm{P_{\eta_{t}} f}_{L^1} \normBig{ \frac{\mathrm{d}\nu_\rho}{\mathrm{d} m_X}}_{L^\infty}\\
&\leq \norm{P_{\eta_{t}} f}_{L^2} \normBig{ \frac{\mathrm{d}\nu_\rho}{\mathrm{d} m_X}}_{L^\infty}\\
&\ll  t^{-c} \cS_{2,1}(f) \rho^{-3 \kappa},
\end{align*}
where the last inequality uses the spectral gap estimate from \Cref{spectralgap-mu} on the one hand,  and  the assumption that $\nu$ is $(1-\kappa, \cB_{\rho}, 0)$-robust on the other hand.  

From the definition of $\nu_{\rho}$, the second term is bounded by
\begin{equation*}
\abse{\int_{X} P_{\eta_{t}} f \dd \nu_\rho - \int_{X} P_{\eta_{t}} f \dd \nu} 
\leq   \rho  \cS_{\infty, 1}(P_{\eta_{t}} f) \ll  \rho t \cS_{\infty, 1}(f).
\end{equation*}

Put together, we have obtained 
\[
| \eta_{t}*\nu(f)| \ll \bigl(t^{-c} \rho^{-3\kappa} +  \rho t) \cS_{\infty,1}(f) \ll \rho^{\kappa} \cS_{\infty,1}(f) 
\]
where the last upper bound holds for $t \in [\rho^{-1/4}, \rho^{-1/2}]$, up to choosing $\kappa = c/16$ 
and $\rho_{0}$ is small enough in terms of $\kappa$.
\end{proof}

We  now address the

\begin{proof}[Proof of \Cref{ThB'}]
Note first that until now, we considered a measure $\lambda$ supported on $\Aff(\R)^+$ while \Cref{ThB'}  allows for a measure $\lambda$ on $\Aff(\R)$. We reduce easily to the $\Aff(\R)^+$-case via the following lemma. 

\begin{lemma} \label{red-Aff+}
We may assume the measure $\lambda$ is supported on $\Aff(\R)^+$. 
\end{lemma}

\begin{proof}
Set $\Omega=\Aff(\R)^{\N}$, consider the stopping time $\tau_{+}:\Omega\rightarrow \N$ defined for $\underline{\phi}=(\phi_{i})_{i\geq 1}\in \Omega$ by
\[
\tau_{+}(\underline{\phi})=\inf\{n\geq 1\,:\,\ttr_{\phi_{1}\circ\dots \circ \phi_{n}}>0\}.
\]
Write $\lambda^{*\tau_{+}}=\int_{\Omega} \delta_{\phi_{1} \circ \dots \circ \phi_{\tau_{+}(\underline{\phi})}} \dd \lambda^{\otimes \N}(\underline{\phi})$.
Then by the strong Markov property, (see \cite[Lemme A.2]{Ben22}), the measure $\sigma$ is $\lambda^{*\tau_{+}}$-stationary. Moreover, $\lambda^{*\tau_{+}}$ has finite exponential moment (because $\tau_{+}$ does).
Its support $\supp \lambda^{*\tau_{+}}$ does not have common fixed point on $\R$, for otherwise the group generated by $\supp \mu$ would have an orbit of cardinality $2$ and hence fixes the barycenter of this orbit.
\end{proof}

Now that $\lambda$ is supported on $\Aff(\R)^+$, we denote by $\mu$ the corresponding measure on $P$. 
We relate the $\eta_{t}$-process  with the $\mu$-random walk thanks to the following lemma.
\begin{lemma}[$\eta_{t}$-process vs $\mu$-walk]\label{lm:cocycle} 
Given  $t>0$, $n\geq 0$,  we have
\[
\eta_t = \int_{P} \eta_{t \ttr_{g}} *\delta_{g} \dd\mu^{*n}(g).
\]
\end{lemma}

\begin{proof}
Observe that for any $s \in \R$ and $g \in P$,
\[
a(t\ttr_g) u(s) g = a(t\ttr_g) u(s) a(\ttr_g)^{-1} u(\ttb_g) = a(t)u( \ttr_g s + \ttb_g).
\]
The claim then follows from the $\lambda^{*n}$-stationarity and the fact that $\mu^{*n}$ and $\lambda^{*n}$ are related by the anti-isomorphism between $P$ and $\Aff(\R)^+$.
\end{proof}

We now discretize the set of values of $\ttr_{g}$ that appears in the part $\eta_{t \ttr_{g}}$. Given $r_{0}, r_{1}>0$ observe that 
\[\eta_{t r_{0}}=\delta_{a(r_{0}r_{1}^{-1})}*\eta_{t r_{1}}\]
Hence, for any finite Borel measure $\nu$ on $X$, we get 
\begin{equation} \label{etat-01}
|\eta_{t r_{0}}*\nu(f)-\eta_{t r_{1}}*\nu(f)|\ll |\log(r_{0}r_{1}^{-1})| \,\nu(X) \cS_{\infty, 1}(f).
\end{equation}
Let $\rho>0$,  consider a parameter $\alpha\in(0, 1)$ to be specified later depending on $\Lambda, \mu$, and set $\sR:=\{(1+\rho^\alpha)^k\,:\, k\in \Z\}$. Combining \eqref{etat-01} with \Cref{lm:cocycle}, we get for any $x\in X$, $f\in B^\infty_{\infty,1}(X)$,  that
\begin{equation}\label{eq-dec1}
\abs{\eta_t* \delta_x (f)} \leq \sum_{r\in \sR} \abs{\eta_{t r} * \mu^n_{r} * \delta_x (f)} +O(\rho^\alpha \cS_{\infty, 1}(f))
\end{equation}
where $\mu^n_{r}$ denotes the restriction of $\mu^{*n}$ to the set $\set{g \in P : \ttr_g \in [r, r(1+\rho^\alpha)[}$.

Let $\kappa=\kappa(\Lambda, \mu)>0$ as in \Cref{endgame}. Assume $\inj(x) \geq \rho$.
By \Cref{high-dim},  there are constants $C=C(\Lambda, \mu) > 1$ and $\eps_1=\eps_{1}(\Lambda, \mu) > 0$ such that, provided $\rho\lll 1$, the measure $\mu^{*n} * \delta_x$ on $X$ is $(1 - \kappa, \cB_\rho, \rho^{\eps_1})$-robust for any $n \geq C \abs{\log \rho}$.
For the rest of this proof, we specify $n,t$ in terms of $\rho$ as 
\begin{equation}
\label{eq:t in rho}
n=\lceil C \abs{\log \rho}\rceil, \,\,\,\,\,\,\,\,\,\,\,\,\,\,\,\,\,\,\,\,t = \rho^{-C \Lyap - 3/8}.
\end{equation}
Consider 
\[
\sR' = \set{r \in \sR : \rho^{-1/4} \leq t r \leq \rho^{-1/2}} = \sR \cap [\rho^{C\Lyap +1/8}, \rho^{C\Lyap  - 1/8}].
\] 

On the one hand, for each $r \in \sR'$, note that $\mu^n_r * \delta_x \leq \mu^{*n} *\delta_x$ is still $(1 - \kappa, \cB_\rho, \rho^{\eps_1})$-robust.
Therefore the choice of $\sR'$ allows us to use \Cref{endgame} to obtain
\begin{equation}\label{eq-dec2}
\abs{\eta_{t r} * \mu^n_r * \delta_x (f)} \leq (\rho^\kappa + \rho^{\eps_1}) \cS_{\infty, 1}(f).
\end{equation}

On the other hand, by the large deviation estimates for sums of i.i.d real random variables, there is a constant $\eps_2 = \eps_2(\mu,C) > 0$ such that 
\[
\mu^{*n} \set{g : \abs{n \Lyap + \log \ttr_g} > \abs{\log \rho}/10} < \rho^{\eps_2}.
\]
whenever $\rho \lll_{C} 1$.
This implies an upper bound on the total mass $\sum_{r \in \sR \setminus \sR'} \mu^n_r(P) \leq \rho^{\eps_2}$ and hence
\begin{equation}\label{eq-dec3}
\sum_{r \in \sR \setminus \sR'} \abs{\eta_{t r} * \mu^n_r * \delta_x (f)} \leq \rho^{\eps_2} \cS_{\infty,1}(f).
\end{equation}

Putting \eqref{eq-dec1}, \eqref{eq-dec2}, \eqref{eq-dec3} together, we have  
\begin{align*}
\abs{\eta_t* \delta_x (f)} 
&\leq  \bigl( \# \sR' (\rho^\kappa + \rho^{\eps_1}) +\rho^{\eps_2}\bigr) \cS_{\infty,1}(f)\\
&\leq \bigl( \rho^{\kappa/2} + \rho^{\eps_{1}/2} +\rho^{\eps_2}\bigr) \cS_{\infty,1}(f)
\end{align*}
where the second bound uses $\sharp \sR' \ll \frac{2C\Lyap |\log \rho|}{\log(1+\rho^\alpha)}\sim |\log \rho| \rho^{-\alpha}$ and assumes $\alpha\leq \kappa\eps_{1}/4$, $\rho\lll 1$.

Viewing $\rho$ as varying with $t$ according to \eqref{eq:t in rho}, we can summarize the above as the following.
For $t > 1$ sufficiently large, for any $x \in X$ with $\inj(x) \geq t^{-(C \Lyap + 3/8)^{-1}}$, 
\[
\abs{\eta_t* \delta_x (f)} \leq  t^{-\eps_3} \cS_{\infty,1}(f)
\]
with $\eps_3 = \frac{\min\{ \kappa, \eps_1, \eps_2\}}{8 C \Lyap + 3}$.
Finally, if $x$ is a point with  $\inj(x) < t^{-(C \Lyap + 3/8)^{-1}}$ then $\inj(x)^{-1} t^{-\eps_3} \geq 1$.
\end{proof}

Effective equidistribution for the $\mu$-walk on $X$ can also be handled similarly (and more simply). 

\begin{proof}[Proof of \Cref{ThC'}]
\Cref{spectralgap-mu} and \Cref{endgame} are still valid with $(t,\eta_{t})$ replaced by $(e^n,\mu^{*n})$ where $n$ is an integer parameter  (essentially same proof, using \Cref{posdim-sigma} instead of \Cref{sigma Holder}).  Combining  \Cref{endgame}  with \Cref{high-dim}, we get the theorem.  More precisely, given $\rho\lll 1$, \Cref{high-dim} tells us that for $m\ggg |\log \rho| +|\log \inj(x)|$, the measure $\nu:=\mu^{*m}*\delta_{x}$
satisfies the conditions required to apply \Cref{endgame}. Then choosing $n>m$ such that $n-m  \in [\frac{1}{4}|\log \rho|, \frac{1}{2}|\log \rho|]$, we obtain that $\mu^{*n}*\delta_{x}$ is $\rho^{\eps}$-equidistributed for some small constant $\eps=\eps(\Lambda, \mu)>0$. This finishes the proof.
\end{proof}

\section{Double equidistribution}  \label{Sec-double-eq}


In this section, we show effective double equidistribution properties for expanding fractals.
This result refines \Cref{ThB'} and will play a  role  in the proof of the divergent case of \Cref{ThA'}. We use the notations set in \Cref{Sec-preliminary}.
In particular, $X=\SL_{2}(\R)/\Lambda$ where $\Lambda$ is an arbitrary lattice,  $x_{0}=\Lambda/\Lambda$ is the basepoint of  $X$, and $\sigma$ is a probability measure on $\R$ that is stationary for a randomized orientation preserving IFS $\lambda$ with a finite exponential moment.

\bigskip
Given a probability measure $\xi$ on $\R$, bounded continuous functions $f_{1},f_{2}:X\rightarrow \R$, and times $t_{2}\geq t_{1}>0$, we introduce  
\begin{equation} \label{double-eq-function}
\Delta^{\xi}_{f_1,f_2}(t_1,t_2) := \abse{\int_\R f_{1}\bigl(a(t_1)u(s)x_{0}\bigr) f_{2} \bigl(a(t_2)u(s)x_{0}\bigr) \dd \xi(s) - m_{X}(f_{1})m_{X}(f_{2})}.
\end{equation}
Hence the probability measure $u(s)x_{0}\dd\xi(s)$ on $X$ enjoys double equidistribution toward $m_{X}$ under expansion by the diagonal flow if for any such $f_{1},f_{2}$, we have 
$$\text{$\Delta^{\xi}_{f_1,f_2}(t_1,t_2)\to 0\,$ as $\,\inf (t_{2}t_{1}^{-1}, t_{1}) \to+\infty$.}$$
In this section, we show that in the case where $\xi=\sigma$, double equidistribution holds with an effective rate.


\begin{proposition}[Effective double equidistribution of expanding fractals] \label{decorrelation-sigma}
For every $\eta>0$, there exist $C, c>0$ such that for all  $ t_{1},t_{2}>1$ with $t_2\geq t^{1+\eta}_{1}$ and $f_{1}, f_{2}\in B^\infty_{\infty,1}(X)$, we have
\begin{equation} \label{double-eq-prop-sigma}
\Delta^{\sigma}_{f_1,f_2}(t_1,t_2) 
\leq C  \cS_{\infty, 1}(f_1) \abs{m_{X}(f_2)}  t_1^{-c} +  C\cS_{\infty, 1}(f_{1})\cS_{\infty, 1}(f_{2}) t_2^{-c}.
\end{equation}
\end{proposition}

Taking $f_{2}=1$ and letting $t_{2}\to+\infty$, we see that \Cref{decorrelation-sigma}  implies   \Cref{ThB'}. 
The proposition assumes  that the times $t_{1}$, $t_{2}$ are slightly separated, via the condition  $ t_2\geq t^{1+\eta}_{1}>1$. In fact we will see later in \Cref{cor-double-+} that \eqref{double-eq-prop-sigma} also implies an upper bound in the \emph{short-range} regime  $t^{1+\eta}_{1}\geq t_2\geq t_{1}$. Namely, for all $t_{2}\geq t_{1}>1$, we have
\begin{multline}
\label{eq:all-range-preview}
\Delta^{\sigma}_{f_1,f_2}(t_1,t_2) 
\leq C \cS_{2,1}(f_1) \cS_{2,1}(f_2) t_1^{c} t_2^{-c} \\ + C \cS_{\infty, 1}(f_1) \abs{m_{X}(f_2)}  t_1^{-c} + C \cS_{\infty, 1}(f_{1})\cS_{\infty, 1}(f_{2}) t_2^{-c},
\end{multline}
for possibly different constants $C,c>0$, depending only on $\Lambda$, $\sigma$. 

The proof of \Cref{decorrelation-sigma} is inspired by  \cite[Theorem 1.2]{KSW17}, which deals with absolutely continuous measures, and \cite[Proposition 10.1]{KL23} which deals with fractal measures and either short-range or long-range regime (i.e. $t_{2}\in [t_{1}, t_{1}^{1+\eps}]$ or $t_{2}\geq t^C_{1}$ where $C^{-1}, \eps\lll 1$). 
Here is the main idea behind the proof.
By self-similarity of $\sigma$,  the distribution $(a(t_{1})u(s)x_{0}, a(t_{2})u(s)x_{0})\dd\sigma(s)$ is roughly that of $(hx_{0}, ghx_{0})\dd\mu^{*n_{1}}(h)\dd\mu^{*n_{2}}(g)$ where $n_{1}\simeq \Lyap^{-1}\log t_{1}$ and $n_{2}\simeq \Lyap^{-1}\log(t_{2}/t_{1})$, with $\Lyap$ the Lyapunov exponent of $\Ad_{\star}\mu$, see \eqref{def-Lyap}. Then we apply  \Cref{ThB'} to the $\mu$-random walk starting at $hx_{0}$, to get that the variable in the second coordinate equidistributes conditionally to the first one. \Cref{ThB'} tells us the first coordinate equidistributes as well, whence the result. 

\begin{proof} 
To lighten notations, we write $\mathcal{S}=\mathcal{S}_{\infty,1}$ and $\cS(f_{1},f_{2})=\cS(f_{1})\cS(f_{2})$.
Noting the relation
\[
\Delta^{\sigma}_{f_1,f_2}(t_1,t_2) \leq \Delta^{\sigma}_{f_1,f_2-m_X(f_2)}(t_1,t_2) +
  \abse{\int_\R f_{1}\bigl(a(t_1)u(s)x_{0}\bigr) \dd \sigma (s) - m_X(f_1)}\abs{m_X(f_2)}
\]
and that \Cref{ThB'} provides us with a constant $c = c(\Lambda, \sigma) > 0$ such that 
\[\abse{\int_\R f_{1}\bigl(a(t_1)u(s)x_{0}\bigr) \dd \sigma (s) - m_X(f_1)} \ll \cS(f_1) t_1^{-c},\]
we can reduce to the case where $m_X(f_2) = 0$.

Thus, we are left to bound the integral
\[I := \int_\R f_{1}\bigl(a(t_1)u(s)x_{0}\bigr) f_{2} \bigl(a(t_2)u(s)x_{0}\bigr) \dd \sigma(s)\]
by a quantity of the form $O_{\eta}(\cS(f_1,f_2)t_2^{-\kappa})$ where $\kappa=\kappa(\Lambda, \sigma, \eta)>0$.

We first  use the $\lambda$-stationarity of $\sigma$ to allow the argument in $f_{2}$ to vary randomly conditionally to that in $f_{1}$. Let $M, n\geq 1$ be  (large) parameters to be specified later. 
By the large deviation principle for sums of i.i.d. real random variables, there exists $\eps=\eps(\lambda, M)>0$ such that, provided $n \ggg_{M} 1$, 
\begin{equation}\label{eq10-2}
\lambda^{*n}(\sC) \geq 1-e^{-n\eps}, \text{ where } \sC := \setbig{\phi \in \Aff(\R)^+ : \abs{n\Lyap+\log \ttr_{\phi}} \leq \frac{n\Lyap}{M}}.
\end{equation}
Note that for any $\phi\in \sC$, $t>1$, $s\in \R$, we have $|\phi(s)-\ttb_{\phi}|\leq \ttr_{\phi}|s|\leq e^{-(1-\frac{1}{M})n\Lyap}|s|$, whence
\begin{equation}\label{eq10-phi(s)-bphi}
|f_1\bigl(a(t) u(\phi(s))x_0\bigr)-f_1\bigl(a(t) u(\ttb_{\phi})x_0\bigr)| \ll t e^{-(1-\frac{1}{M})n\Lyap}|s|\mathcal{S}(f_1).
\end{equation}
Moreover, as $s$ varies with law $\sigma$, its size is controlled by the moment estimate of \Cref{sigma Holder}. Namely, there is some $\gamma=\gamma(\sigma)>0$ such that for all $R>1$,
\begin{equation}\label{eq10-1}
\sigma \set{s\in \R : |s| >R}\ll R^{-\gamma}.
\end{equation}
Splitting the integral on $s$ according to whether $\abs{s} \leq e^{\frac{n\Lyap}{M}}$ or not and using \eqref{eq10-phi(s)-bphi} and \eqref{eq10-1}, we obtain
\begin{equation}\label{eq10-3}
\int_{\R} \bigl\lvert f_1\bigl(a(t) u(\phi(s))x_0\bigr) - f_1\bigl(a(t) u(\ttb_{\phi})x_0\bigr) \bigr\rvert \dd\sigma(s)\ll (e^{-\frac{n\Lyap \gamma}{M}}+te^{-(1-\frac{2}{M})n\Lyap})\mathcal{S}(f_1).
\end{equation}
Using the $\lambda$-stationarity of $\sigma$, then applying  \eqref{eq10-2} followed by \eqref{eq10-3}, we deduce
\begin{align*}
I &= \int_{\Aff(\R)^+}\int_{\R} f_{1}\bigl(a(t_1)u(s)x_{0}\bigr) f_{2} \bigl(a(t_2)u(s)x_{0}\bigr) \dd (\phi_{\star}\sigma)(s) \dd\lambda^{*n}(\phi)\\
&=  \int_{\sC} f_{1}\bigl(a(t_1)u(\ttb_{\phi})x_{0}\bigr)  \int_{\R} f_{2} \bigl(a(t_2)u(s)x_{0}\bigr) \dd (\phi_{\star}\sigma)(s) \dd\lambda^{*n}(\phi) +E_{1}
\end{align*}
where $E_{1}$ stands for the error term $E_{1}=O_{M}(e^{-\frac{n\Lyap \gamma}{M}}+t_{1}e^{-(1-\frac{2}{M})n\Lyap}+e^{-n\eps})\cS(f_{1}, f_{2})$.

It follows that 
\begin{equation}
\label{I-f1-const}
\abs{I} \leq \cS(f_1) \int_{\sC} \abse{\int_{\R} f_{2} \bigl(a(t_2)u(s)x_{0}\bigr) \dd (\phi_{\star}\sigma)(s)} \dd\lambda^{*n}(\phi) + |E_{1}|
\end{equation}

The inner integral invovling $f_{2}$ can be bounded using \Cref{ThB'}.
Indeed, recall that $m_X(f_2) = 0$ and that for any $t>0$ and any $s\in \R$, 
\[a(t) u(\phi(s))=a(t\ttr_{\phi})u(s)h_{\phi}\]
where $h_{\phi}=a(\ttr_{\phi}^{-1})u(\ttb_{\phi})$.
Hence, for any $\phi \in \sC$,
\begin{align}\label{eq10-4}
\int_{\R} f_{2} \bigl(a(t_2)u(s)x_{0}\bigr) \dd (\phi_{\star}\sigma)(s) &=\int f_2\bigl(a(t_{2}\ttr_{\phi})u(s)h_{\phi}x_0\bigr) \dd\sigma(s) \nonumber\\
&= O\bigl(\inj(h_{\phi}x_0)^{-1}\mathcal{S}(f_2) t_{2}^{-c} e^{c(1+1/M) n\Lyap}\bigr).
\end{align}
where $c=c(\Lambda, \sigma)>0$ is the exponent provided by \Cref{ThB'}.
Note that $h_{\phi}$ has law $\mu^{*n}$ when $\phi$ varies randomly according to $\lambda^{*n}$.
Hence, by the effective recurrence of the $\mu$-random walk on $X$ (\Cref{effective-recurrence}), there exists $\delta=\delta(\Lambda, \lambda)>0$ such that 
\begin{align}\label{eq10-5}
\lambda^{*n}\setbig{\phi \in \Aff(\R)^+ : \inj(h_{\phi}x_0)\leq e^{-\frac{c}{M}n\Lyap}} \ll e^{- \delta \frac{c}{M} n\Lyap}. 
\end{align}
Note that for $\phi\in \sC$ not belonging to the set in \eqref{eq10-5}, the error term in \eqref{eq10-4}   is bounded by $O(\mathcal{S}(f_2) t_{2}^{-c} e^{c(1+2/M)n\Lyap} )$.
Therefore, we see from \eqref{I-f1-const}, \eqref{eq10-4}, \eqref{eq10-5} that
\begin{equation*}
\abs{I} \ll (t_{2}^{-c} e^{c(1+2/M)n\Lyap} + e^{- \delta \frac{c}{M} n\Lyap})\cS(f_{1}, f_{2})+|E_{1}|.
\end{equation*}

Recalling the value of $E_1$ and choosing $n$ such that $n\Lyap= \frac{1}{2}\log t_1 + \frac{1}{2}\log t_2 + O(1)$, we obtain
\[
\abs{I} \ll_{M} \cS(f_{1}, f_{2}) \Bigl((t_{2}/t_{1})^{-c/2}(t_{1}t_{2})^{c/M} + (t_{1}t_{2})^{-c'}+ (t_{2}/t_{1})^{-1/2}(t_{1}t_{2})^{1/M} \Bigr)
\]
where $c'>0$ only depends on  $\Lambda$, $\lambda$, $\sigma$, $M$. 
The desired estimate $\abs{I} \ll t_2^{-\kappa}$ follows, provided $M$ has been chosen large enough from the start depending on the separation parameter $\eta$.
\end{proof}

\section{The dichotomy} \label{Sec-Khintchine-dich}

We show that an arbitrary probability measure $\xi$ on $\R$ obeys the Khintchine dichotomy provided that the pushfoward $a(t)u(s)\SL_{2}(\Z)\dd\xi(s)$ exhibits certain effective equidistribution properties on $\SL_{2}(\R)/\SL_{2}(\Z)$ for large $t$. We deduce \Cref{ThA'}  (whence \Cref{Kintchine-Cantor}). 
We use the notations introduced in \Cref{Sec-preliminary}.

\begin{definition} Let $\xi$ be a probability measure on $\R$. 
We say that $\xi$ satisfies the \emph{effective single equidistribution property on $X$} if there are constants $C,c>0$ such that
\begin{multline}\label{single-eq-prop}
\forall f \in B^\infty_{\infty,1}(X),\, \forall t > 1,\\
\abse{\int_\R f\bigl(a(t) u(s) x_0\bigr) \dd \xi(s) - m_{X}(f) } \leq C \cS_{\infty, 1}(f) t^{-c}.
\end{multline}

We say that $\xi$ satisfies the \emph{effective double equidistribution property on $X$} if for any $\eta > 0$, there are constants $C,c>0$ such that
\begin{multline}\label{double-eq-prop}
\forall f_{1},f_{2} \in B^\infty_{\infty,1}(X),\,  \forall t_{1} > 1,\, \forall t_{2}> t_{1}^{1+\eta},\\
\Delta^{\xi}_{f_1,f_2}(t_1,t_2) 
\leq  C\cS_{\infty, 1}(f_1) \abs{m_{X}(f_2)}  t_1^{-c} +  C\cS_{\infty, 1}(f_{1})\cS_{\infty, 1}(f_{2}) t_2^{-c}.
\end{multline}
where the notation $\Delta^{\xi}_{f_1,f_2}(t_1,t_2) $ is defined in \eqref{double-eq-function}. See \Cref{cor-double-+} for an alternative characterization.
\end{definition}

In \cite{KL23}, Khalil-Luethi showed that effective single equidistribution  implies the convergent case of the Khintchine dichotomy.  

\begin{thm}[Convergent case  {\cite[Theorem 9.1]{KL23}}] \label{convergent-Khintchine}
Let $\xi$ be a probability measure on $\R$ satisfying the effective single equidistribution property \eqref{single-eq-prop} on $\SL_{2}(\R)/\SL_{2}(\Z)$. Then for every non-increasing  function $\psi: \mathbb{N}\to \R_{> 0}$  such that $\sum_q \psi(q)<\infty $, we have  
$$\xi (W(\psi))=0. $$ 
\end{thm}

We show that effective double equidistribution implies the divergent case of the Khintchine dichotomy. 
Moreover our method yields  quantitative estimates on the number of solutions of the Diophantine inequality when bounding the denominator.  
We set $\cP(\Z^2) := \set{(p,q) \in \Z^2 : \gcd(p,q) = 1}$  the set of {primitive} elements in $\Z^2$. We let $\zeta(t)=\sum_{n\geq 1} n^{-t}$ denote the Riemann zeta function.

\begin{thm}[Divergent case]  \label{quant-Khintchine-1/q}
Let $\xi$ be a  probability measure on $\R$ satisfying the effective double equidistribution property \eqref{double-eq-prop} on $\SL_{2}(\R)/\SL_{2}(\Z)$.
Let $\psi: \mathbb{N}\to \R_{> 0}$ be a non-increasing function satisfying $\sum_q \psi(q)=\infty $, as well as 
\begin{equation}
\label{eq:leqoverq}
\forall q \in \N,\quad \psi(q)\leq q^{-1}.
\end{equation} 

Then for $\xi$-almost every $s\in \R$,  as $N\to +\infty$, we have 
\begin{equation}
\label{eq:primsol}
    \#\set{(p,q)\in \mathcal{P}(\Z^2):  \,1\leq q\leq N, \,0\leq qs-p < \psi(q)} \,\sim_{\xi, \psi,s}\,   \zeta(2)^{-1} \sum_{q=1}^N \psi(q).
\end{equation}
The same  holds if we ask for $- \psi(q) < qs-p \leq 0$ instead.
\end{thm}

Without the extra domination assumption \eqref{eq:leqoverq} on the approximation function $\psi$, we still have a quantitative lower bound (which tends to infinity).
\begin{corollary}  \label{quant-Khintchine}
If $\xi$ satisfies \eqref{double-eq-prop} on $\SL_{2}(\R)/\SL_{2}(\Z)$ and $\psi: \mathbb{N}\to \R_{> 0}$ is non-increasing with  $\sum_q \psi(q)=\infty $, then  for $\xi$-almost every $s\in \R$,  as $N\to +\infty$, we have 
\begin{equation} \label{quant-khint-eq}
    \#\set{(p,q)\in \mathcal{P}(\Z^2):  \,1\leq q\leq N, \,0\leq qs-p < \psi(q)}\, \geq \,(1+o_{\xi,\psi,s}(1))   \zeta(2)^{-1} \sum_{q=1}^N \min(\psi(q),  q^{-1}).
\end{equation}
The same  holds if we ask for $- \psi(q) < qs-p \leq 0$ instead.
\end{corollary}

Assuming \Cref{quant-Khintchine-1/q}, we establish \Cref{quant-Khintchine} and \Cref{ThA'}.

\begin{proof}[Proof of \Cref{quant-Khintchine}] It follows from \Cref{quant-Khintchine-1/q} applied to the approximation function $q\mapsto \min(\psi(q),  q^{-1})$. Indeed, this application is allowed because we have  $\sum_{q} \min(\psi(q),  q^{-1})=\infty$. To justify this, observe that given any non-increasing function $\Psi :\N\rightarrow \R^+$, we have $\sum_{q} \Psi(q)=\infty$ if and only if $\sum_{n} 2^n\Psi(2^n)=\infty$. Hence $\sum_{q} \min(\psi(q),  q^{-1})=\infty$ amounts to $\sum_{n} \min(2^n\psi(2^n), 1)=\infty$ which in turns follows from $\sum_{n} 2^n\psi(2^n)=\infty$. 
\end{proof}

\begin{proof}[Proof of  \Cref{ThA'}]
As in the proof of \Cref{ThB'}, we may assume $\lambda$ {is}  supported on $\Aff(\R)^+$, see \Cref{red-Aff+}. Hence we are reduced to the setting of \Cref{Sec-preliminary}. 
By \Cref{decorrelation-sigma}, $\sigma$ satisfies the effective double equidistribution property  \eqref{double-eq-prop} (and in particular  \eqref{single-eq-prop}). Hence both \Cref{convergent-Khintchine} and \Cref{quant-Khintchine} apply, yielding the announced dichotomy. 
\end{proof}

We now pass to the proof of \Cref{quant-Khintchine-1/q}. In a first step we will show that effective double equidistribution in fact yields  decorrelation estimates that are valid for all times $t_{2}\geq t_{1}\geq1$. Then we will exploit these  estimates through the mean of Dani's correspondence to deduce the theorem.

\subsection{Single vs double equidistribution} 
Note that effective double equidistribution  \eqref{double-eq-prop} implies effective single equidistribution \eqref{single-eq-prop}. Conversely, effective single equidistribution  gives a double equidistribution estimate in the short-range regime. The proof exploits  the decay of matrix coefficients as in \cite[Theorem 10.1]{KL23}. 

\begin{lemma}\label{lm:short-range-decorr}
Let $\xi$ be a Borel probability measure on $\R$ satisfying \eqref{single-eq-prop} with associated constants $C> 1, c\in (0,1)$. Then for every $t_1, t_{2} \geq 1$ such that $t^{1+c/2}_{1}\geq t_{2}\geq t_{1}$ and every $f_1, f_2 \in B^\infty_{\infty,1}(X)$,
\begin{equation}
\label{eq:short-range}
\Delta^{\xi}_{f_1,f_2}(t_1,t_2) 
\ll  \cS_{2,1}(f_1) \cS_{2,1}(f_2) t_1^{\delta_0} t_2^{-\delta_0}  + C \cS_{\infty,1}(f_1) \cS_{\infty,1}(f_2) t^{-c/3}_{2}.
\end{equation}
where $\delta_0=\delta_{0}(\Lambda)>0$ arises from \eqref{eq:decay cor}.
\end{lemma}

\begin{proof}
Let $t_2\geq t_{1}\geq 1$ and $f_1,f_2\in B^\infty_{\infty,1}(X)$.
Set $F : X \to \R$ to be 
\[
F(x) = f_1(x) f_2(a(t_2/t_1)x), \quad x \in X,
\]
so that $F\bigl(a(t_1) u(s) x_0\bigr) = f_1\bigl(a(t_1) u(s) x_0\bigr) f_2\bigl(a(t_2) u(s) x_0\bigr)$ for all $s \in \R$.
Then 
\begin{align*}
\cS_{\infty,1}(F) &\ll \cS_{\infty,1}(f_1) \cS_{\infty,1}(a(t_1/t_2).f_2) \ll \cS_{\infty,1}(f_1) \norm{\Ad(a(t_1/t_2))}\cS_{\infty,1}(f_2) \\
&\ll t_2/t_1  \cS_{\infty,1}(f_1) \cS_{\infty,1}(f_2).
\end{align*}
By \eqref{single-eq-prop} applied to $F$ and $t_1$,
\[
\abse{\int_\R F\bigl(a(t_1) u(s) x_0\bigr) \dd \xi(s) - \langle f_1, a(t_1/t_2).f_2 \rangle } \leq C \cS_{\infty, 1}(F) t_1^{-c}
\]
By \eqref{eq:decay cor}, we have
\[
\abse{\langle f_1, a(t_1/t_2).f_2\rangle - m_X(f_1) m_X(f_2)} \ll \norm{a(t_1/t_2)}^{-\delta_0} \cS_{2,1}(f_1) \cS_{2,1}(f_2).
\]
Combining the above together, we obtain 
$$\Delta^{\xi}_{f_1,f_2}(t_1,t_2) 
\ll  \cS_{2,1}(f_1) \cS_{2,1}(f_2) t_1^{\delta_0} t_2^{-\delta_0}  + C \cS_{\infty,1}(f_1) \cS_{\infty,1}(f_2) t_{2}t_{1}^{-1-c}$$
whence the desired inequality in the regime $t^{1+c/2}_{1}\geq t_{2}\geq t_{1}$.
\end{proof}

We deduce that even though double equidistribution was formulated with the separation assumption $t_{2}> t^{1+\eta}_{1}$ on the parameters $t_{1},t_{2}\geq 1$, it still provides  estimates in the short-range regime $t_{2}\leq t^{1+\eta}_{1}$. Put together, we obtain the following result.

\begin{corollary} \label{cor-double-+}
A probability measure $\xi$ on $\R$ satisfies the effective double equidistribution property \eqref{double-eq-prop} if and only if  there exist constants $c > 0$ and $C > 1$ such that for every $f_1,f_2 \in B^{\infty}_{\infty,1}(X) $ and all $t_2 \geq t_1 > 1$.
\begin{multline}
\label{eq:all-range}
\Delta^{\xi}_{f_1,f_2}(t_1,t_2) 
\leq C \cS_{2,1}(f_1) \cS_{2,1}(f_2) t_1^{c} t_2^{-c} \\ + C \cS_{\infty, 1}(f_1) \abs{m_{X}(f_2)}  t_1^{-c} + C \cS_{\infty, 1}(f_{1})\cS_{\infty, 1}(f_{2}) t_2^{-c},
\end{multline}
\end{corollary}

\begin{proof}
As the $\cS_{2,1}$-norm is bounded by the $\cS_{\infty,1}$-norm, the converse direction is clear. Assume  $\xi$ satisfies \eqref{double-eq-prop}.
Recalling that \eqref{double-eq-prop} implies \eqref{single-eq-prop}, \Cref{lm:short-range-decorr} applies and yields the upper bound  in the short-range regime (possibly with different values of $C,c$). It also holds in the non short-range regime by definition of \eqref{double-eq-prop}.
\end{proof}

\subsection{Lower bound estimate} 
In this subsection, we establish the lower bound in our quantitative Khintchine dichotomy  \Cref{quant-Khintchine-1/q}. Notations refer to \Cref{quant-Khintchine-1/q}, in particular $X$ {here is}  $\SL_{2}(\R)/\SL_{2}(\Z)$ and $\psi(q)\leq q^{-1}$. We extend $\psi$ to a function $\R^+ \to \R_{>0}$ by setting $\psi(q) = \psi(\lceil q \rceil)$ for non-integer values of $q$, so that it is still non-increasing and smaller than $q^{-1}$.

For $N \geq 1$ and $s \in \R$, we write $\sT_N(s)$ for the left-hand side of \eqref{eq:primsol}, on which we aim to obtain a lower bound.
We fix a parameter $\tau \in (1, 2]$ and define for $k \geq 0$, 
\[
\sS_{k}(s) := \#\set{(p,q)\in \mathcal{P}(\Z^2)\,:\,  \tau^{k-1} < q \leq \tau^k,\, 0\leq qs - p < \psi(\tau^k)}.
\]
Letting $n \geq 1$ be such that $\tau^n \leq N < \tau^{n+1}$ and using that $\psi$ is non-increasing, we have 
\begin{equation}
\label{eq:sTsS}
\sT_N(s) \geq \sT_{\tau^n}(s) \geq \sum_{k=1}^n \sS_{k}(s).
\end{equation}

We bound below  the sum on the right hand side.
\begin{proposition}\label{pr:rough}
Under the assumptions of \Cref{quant-Khintchine-1/q},  for $\xi$-almost all $s \in \R$, for every $\eps > 0$, for all large enough $n$, we have 
\begin{equation}
\label{eq:sumSk}
\sum_{k=1}^n \sS_{k}(s) \geq (1-\eps) \zeta(2)^{-1}\sum_{k=1}^n (\tau^k - \tau^{k-1}) \psi(\tau^k). 
\end{equation}
\end{proposition}

The lower bound in \Cref{quant-Khintchine-1/q} follows at once. 
\begin{proof}[Proof of lower bound in \eqref{eq:primsol} using \Cref{pr:rough}]
In view of \eqref{eq:sTsS} and \Cref{pr:rough}, it suffices to show that for any $\eps > 0$, there is some $\tau > 1$ such that 
\[
\sum_{k=1}^n (\tau^k - \tau^{k-1}) \psi(\tau^k) \geq (1 - 3 \eps) \sum_{q=1}^N \psi(q)
\]
whenever $\tau^n \leq N < \tau^{n+1}$ and $N$ is large enough (in terms of $\psi$ and $\eps$).

Indeed, we can pick $\tau = 1 + \eps$. Because $\psi$ is non-increasing, we have
\[
\sum_{q= \lceil \tau^k \rceil}^{\lceil \tau^{k+1} \rceil - 1} \psi(q) \leq (\lceil \tau^{k+1} \rceil - \lceil \tau^{k} \rceil)\psi(\tau^k) \leq (\tau + \eps) (\tau^k - \tau^{k-1}) \psi(\tau^k)
\]
for all $k \geq 1$ large enough.
Summing up to $k=n$ yields the desired inequality.
\end{proof}
\smallskip

We now turn to the proof of \Cref{pr:rough}.

\smallskip
First, we invoke Dani's correspondence to give the quantity $\sS_k(s)$ a dynamical interpretation.
Consider $X = \SL_2(\R) /\SL_2(\Z)$ and  $x_0 = \SL_2(\Z)/\SL_2(\Z) \in X$  the identity coset.
For a function $f : \R^2 \to [0, +\infty]$, we denote by $\widetilde{f} : X \to [0, +\infty]$ its primitive Siegel transform. It is defined by: $\forall g \in G$,
\[
\widetilde{f}(gx_0) = \sum_{v \in \cP(\Z^2)} f(g v).
\]
 For each $k \geq 1$, consider the quantities $r_k, t_k \in \R_{>0}$ such that 
\[\tau^k= r_k t_k^{1/2},\qquad \psi(\tau^k)= r_k t_k^{-1/2},\]
or equivalently
\begin{equation} \label{def-tk-rk}
r_k^{2}= \tau^k \psi(\tau^k),\qquad  t_k = \tau^k\psi(\tau^k)^{-1}.
\end{equation}
Consider the rectangle $R_k = [0,r_k) \times (\tau^{-1}r_k,r_k] \subset \R^2$.
Direct computation shows that for any $s \in \R$,
\[
\sS_{k}(s)= \widetilde\1_{R_k}\bigl(a(t_{k})u(s)x_{0}\bigr).
\]

Next, we construct  smooth lower approximations $(\varphi_k)_{k\geq1}$ of the functions $(\widetilde\1_{R_k})_{k\geq1}$. This substitution will allow us to use  equidistribution estimates.
Let $\eps > 0$ be a (small) parameter.
Let $R_k^- := \bigl[\eps r_k,(1 - \eps)r_k\bigr) \times \bigl((\tau^{-1} + \eps)r_k, (1 - \eps)r_k\bigr] \subset R_k$ denote the rectangle shrunken by $\eps r_k$ on each side of $R_k$.
Note that for every $g \in B_{\eps/10} \subset G$, we have $g R_k^- \subset R_k$ and hence 
$g_* \widetilde\1_{R_k^-} \leq \widetilde\1_{R_k}$.
Let $\theta_\eps : G \to \R^+$ be a smooth bump function supported on $B_{\eps/10}$ such that $m_G(\theta_\eps) = 1$ and $\cS_{\infty,1}(\theta_\eps) \ll \eps^{-4}$. 
We set for every $k \geq 1$,
\[
\varphi_k := \theta_\eps * \widetilde\1_{R_k^-}.
\]
In particular, $\varphi_k\leq \widetilde\1_{R_k}$, so  for every $s \in \R$,
\begin{equation}
\label{eq:YkleqSk}
\varphi_k\bigl(a(t_{k})u(s)x_{0}\bigr) \leq \sS_{k}(s).
\end{equation}

We now discuss the norm properties of the functions $\varphi_{k}$.

\begin{lemma} For every $k\geq 1$, we have
\begin{align}
&m_X(\varphi_k) =\zeta(2)^{-1} r_k^2 (1 - 2\eps)(1 - \tau^{-1} - 2 \eps) \label{L1norm-varphik}.\\
&\cS_{\infty,1}(\varphi_k)  \ll \eps^{-1} \,\,\,\,\,\,\,\,\,\,\,\,\,\,\,\,
\cS_{2,1}(\varphi_k) \ll \eps^{-1} \sqrt{m_X(\varphi_k)}.  \label{eq:controlSobolev}
\end{align}
\end{lemma}

\begin{proof}
Note that by our assumption on $\psi$, we have $r_k \leq 1$ hence $R_k \subset [0,1) \times [0,1]$ contains at most $2$ primitive vectors of any unimodular lattice in $\R^2$.
It follows that
\[
\norm{\widetilde\1_{R^-_k}}_{L^\infty} \leq \norm{\widetilde\1_{R_k}}_{L^\infty} \leq 2
\]
and then $\norm{\varphi_k}_{L^\infty} \leq 2$.
We recall here that we use the \emph{primitive} Siegel transform. Had we used the non-primitive version of the Siegel transform, the norm $\norm{\widetilde\1_{R^-_k}}_{L^\infty}$ would not be finite.

By Siegel's summation formula~\cite[Equation 25]{Siegel},
\begin{align*}
m_X(\varphi_k) &= m_G(\theta_\eps) m_X(\widetilde\1_{R_k^-}) = \zeta(2)^{-1} \Leb_{\R^2}(R^-_{k}) \nonumber\\
&= \zeta(2)^{-1} r_k^2 (1 - 2\eps)(1 - \tau^{-1} - 2 \eps). \label{eq:mXphik}
\end{align*}
Then
\begin{equation*}
\label{eq:controlSinfty1}
\cS_{\infty,1}(\varphi_k) \leq m_G(\supp \theta_\eps) \cS_{\infty,1}(\theta_\eps)  \norm{\widetilde\1_{R_k^-}}_{L^\infty} \ll \eps^{-1}.
\end{equation*}
Finally, using that $\widetilde\1_{R_k}$ only takes  integer values,
\begin{multline*}
\label{eq:controlS21}
\cS_{2,1}(\varphi_k) \leq \cS_{\infty,1}(\varphi_k) \sqrt{m_X(\supp \varphi_k)} \\
\ll \eps^{-1}\sqrt{m_X(\supp \widetilde\1_{R_k})} 
\leq \eps^{-1} \sqrt{m_X(\widetilde\1_{R_k})}  \ll \eps^{-1} \sqrt{m_X(\varphi_k)}. 
\end{multline*}
where the last bound relies on \eqref{L1norm-varphik} and Siegel's summation formula again. 
\end{proof}


We consider $(\R,\xi)$ as a probability space.
Expectation $\E[\,\cdot\,]$ refers implicitely to this probability space.
Introduce for every $k \geq 1$, the random variable
\[ Y_k : \R \to \R, \quad s \mapsto \varphi_k\bigl(a(t_{k})u(s)x_{0}\bigr).\] 
Write
\[
y_k = m_X(\varphi_k) \in [0,1]
\]
and set $Z_k = Y_k - y_k$ as the (quasi-recentered) companion of $Y_k$. 

From the quantitative double equidistribution hypothesis on $\xi$, we deduce an upper bound on the second moment of a sum  of $Z_k$'s. 
\begin{proposition} \label{L2bound-Zj}
In the setting of \Cref{quant-Khintchine-1/q}, assume additionally that
\begin{equation}
\label{eq:qlogq}
\psi(q) \geq q^{-1}\log^{-2}(q), \quad  \forall q \geq3.
\end{equation}
Then there is a constant $C'$ such that for every subset $J \subseteq \N^*$ we have
\[\E \Bigl[\Bigl(\sum\nolimits_{j \in J}  Z_{j}\Bigr)^2\Bigr] \leq C' \sum_{j\in J} y_j.\]
\end{proposition}

\begin{proof}
Let $C > 1$ and $c > 0$ be as in the full-range double equidistribution estimate \eqref{eq:all-range}.
In this proof, the implied constants in the $\ll$ notation are allowed to depend on  $C$, $\tau$ and $\eps > 0$.

By definition, for each $k,l \geq1$,
\begin{equation*}
\E[Z_{k}Z_{l}] = \E[Y_{k}Y_{l}] - y_k y_l - \E[Z_k]y_l - y_k\E[Z_l].
\end{equation*}
Combining \eqref{eq:all-range} with the bounds on Sobolev norms from \eqref{eq:controlSobolev}, we obtain for $ k \leq l$,
\[
\abs{\E[Y_{k}Y_{l}] - y_k y_l} \ll \sqrt{y_k y_l} t_k^{c} t_l^{-c} +  y_l t_k^{-c} + t_l^{-c}.
\]
while by \eqref{single-eq-prop}, 
\begin{equation*}
\abs{\E[Z_k]} \ll t_k^{-c}.
\end{equation*}
By expanding the square power, using the above bounds, and recalling from  \eqref{def-tk-rk} that  $t_k \geq \tau^k$ and $t_l/t_k \geq \tau^{l-k}$ for $k\leq l$, we deduce 
\[\E \Bigl[\Bigl(\sum\nolimits_{j \in J}  Z_{j}\Bigr)^2\Bigr] \ll\sum\nolimits_{k,l \in J, k\leq l} (\sqrt{y_k y_l} \tau^{-c(l-k)} +  y_l \tau^{-ck} + \tau^{-cl}).\]
Using $\sqrt{y_k y_l}\leq y_{k}+y_{l}$ and the convergence $\sum_{n = 0}^\infty \tau^{-cn} < + \infty$, the first sum satisfies $\sum\nolimits_{k,l \in J, k\leq l} \sqrt{y_k y_l} \tau^{-c(l-k)}\ll \sum y_{j}$. The convergence  $\sum_{n = 0}^\infty \tau^{-cn} < + \infty$ bounds similarly the second sum. To bound the third sum, note that combining  \eqref{def-tk-rk}  with our assumption \eqref{eq:qlogq}, then using \Cref{L1norm-varphik},  we have
\[
\tau^{-c k}  \ll (k\log \tau)^{-2} \leq r_k^2 \ll y_k.
\]
Hence $\tau^{-cl} \ll y_k \tau^{-c(l-k)}$, so   $\sum\nolimits_{k,l \in J, k\leq l} \tau^{-cl} \ll \sum y_{j}$ as for the first sum.
\end{proof}

The following lemma is a general fact about sequences of random variables. It is abstracted  from Schmidt's proof of the quantitative Khintchine theorem for the Lebesgue measure~\cite{Schmidt}.
See also \cite[Chapter I, Lemma 10]{Sp79}, or \cite[Lemma 2.6]{KM99}.

\begin{lemma}\label{Lemma: Schmidt's argument}
Let $(Y_k)_{k\geq1}$ be a sequence of non-negative real random variables.
Let $(y_k)_{k\geq1} \in [0,1]^{\N^*}$ be a sequence of real numbers, set $Z_k=Y_k-y_k$.
Assume that $\sum_{k=1}^{\infty} y_k = +\infty$ and for some $C_1 \geq 1$
\begin{equation}\label{qeq-2}
\forall n \geq m \geq 1, \quad  \E\Bigl[\bigl(\sum_{k=m}^n Z_k\bigr)^2\Bigr]\leq C_1 \sum_{k=m}^n y_k.
\end{equation}
Then almost surely, for large enough $n$, we have 
\[
\Bigl\lvert \sum_{k=1}^n Z_k  \Bigr\rvert \leq  \Bigl(\sum_{k=1}^n y_k\Bigr)^{1/2} \log^{2}\Bigl(\sum_{k=1}^n y_k\Bigr).
\]
\end{lemma}

We are now able to conclude the proof of  \Cref{pr:rough}, whence that of the lower bound in \Cref{quant-Khintchine-1/q}. 

\begin{proof}[Proof of \Cref{pr:rough}]
The series $\sum_q q^{-1}\log^{-2}(q)$ is convergent.
Thus, by the convergent case of the Khintchine dichotomy for measures satisfying~\eqref{single-eq-prop} (\Cref{convergent-Khintchine}), we know that if we replace $\psi$ by $q \mapsto \max\{\psi(q), q^{-1}\log^{-2}(q)\}$ (say for $q\geq 3$, and by $q\mapsto 1/2$ else), then for $\xi$-almost every $s \in \R$, the left-hand side of \eqref{eq:sumSk} is increased by only a bounded amount.
For this reason, without loss of generality, we can assume \eqref{eq:qlogq}.

Note that in view of the inequality \eqref{eq:YkleqSk}, we have $\sum_{k=1}^n \mathscr{S}_{k} \geq \sum_{k=1}^n Y_{k}$. 
Equations \eqref{L1norm-varphik} and \eqref{def-tk-rk} yield $1\geq y_{k}\geq \zeta(2)^{-1}(1-O(\eps))(\tau^k-\tau^{k-1})\psi(\tau^k)$, in particular $\sum_{k=1}^\infty y_{k}=\infty$. 
This estimate, combined with the previous paragraph and the variance bound \Cref{L2bound-Zj}, allows to apply  \Cref{Lemma: Schmidt's argument} to get that $\xi$-almost everywhere, $\sum_{k=1}^n Y_{k}\sim \sum_{k=1}^n y_{k}\geq \zeta(2)^{-1}(1-O(\eps))\sum_{k=1}^n (\tau^k-\tau^{k-1})\psi(\tau^k)$. This concludes the proof.
\end{proof}

\subsection{Upper bound estimate} \label{Sec-upp-bd}
The proof of the upper bound in \Cref{quant-Khintchine-1/q} (\Cref{eq:primsol}) is similar. We extend $\psi$ to $\R^+$ by setting $\psi(q)=\min(q^{-1}, \psi(\lfloor q\rfloor)$ for non-integer values of $q$.
We have for $\tau^n \leq N < \tau^{n+1}$ and for every $s \in \R$,
\[
\sT_N(s) \leq \sum_{k=0}^n \sS_k^+(s)
\]
where for every $k \geq 0$,
\[
\sS_k^+(s) := \#\set{(p,q)\in \mathcal{P}(\Z^2)\,: \, \tau^{k} \leq q < \tau^{k+1}, \,0\leq qs - p < \psi(\tau^k)}.
\]
Then
\[
\sS_k^+(s) = \widetilde{\1}_{ [0,r_k) \times [r_k,\tau r_k)} \bigl( a(t_k)u(s)x_0 \bigr) \leq \varphi_k^+\bigl( a(t_k)u(s)x_0 \bigr)
\]
where $\varphi_k^+ = \theta_\eps * \widetilde\1_{R_k^+}$ with $\eps\in (0, 1)$ small and  
\[R_k^+=[-\eps r_k,(1+\eps)r_k) \times [(1-\eps)r_k,(\tau + \eps) r_k).\]
Note that  $\psi(q)\leq 1/q$ implies  $r_k \leq 1$, so $R_k^+$ is contained in the ball of radius $4$ centered at $0\in \R^2$. Hence, $\norm{\widetilde\1_{R_k^+}}_{L^{\infty}}$ is  bounded above by an absolute constant (independently of $k$).
For the rest of the proof, we can use mutatis mutandis the argument for the lower bound.

\bibliographystyle{abbrv} 
\bibliography{khintchine}
\end{document}